\documentclass{siamltex}

\usepackage{amssymb}
\usepackage{amssymb}
\usepackage{amsmath}
\usepackage[dvips]{graphicx}
\DeclareGraphicsExtensions{.eps}
\usepackage{graphicx}
\usepackage{color}
\usepackage{subfigure}

\def\xkd{x_k^{\delta}}
\def\xk+1d{x_{k+1}^{\delta}}

\newtheorem{ipotesi}{Assumption}[section]

\newcommand{\beqn}[1]{\begin{equation}\label{#1}}
\newcommand{\eeqn}{\end{equation}}
\renewcommand{\theequation}{\arabic{section}.\arabic{equation}}
\newcommand{\req}[1]{(\ref{#1})}

\newcounter{algo}[section]
\renewcommand{\thealgo}{\thesection.\arabic{algo}}
\newcommand{\algo}[3]{\refstepcounter{algo}
\begin{center}\begin{figure}[htbf]
\framebox[\textwidth]{
\parbox{0.95\textwidth} {\vspace{\topsep}
{\bf Algorithm \thealgo : #2}\label{#1}\\
\vspace*{-\topsep} \mbox{ }\\
{#3} \vspace{\topsep} }}
\end{figure}\end{center}}

\newcommand{\yd}{y^{\delta}}
\newcommand{\Fdk}{F(\xkd)-y^{\delta}}
\newcommand{\Jdk}{J(\xkd)}
\def\mklm{m_k^{\mathrm{LM}}}
\def\mktr{m_k^{\mathrm{TR}}}
  
\newcommand{\comment}[1]{}
\newcommand{\z}{\phantom{0}}

\begin{document}
%
%
\newpage

\title{On an adaptive regularization for ill-posed nonlinear systems
and its trust-region implementation
\thanks{Work partially supported by INdAM-GNCS, under the 2015 Projects 
``Metodi di regolarizzazione  per problemi di ottimizzazione
e applicazioni''}}
\author{S. Bellavia\footnotemark[2] and
B. Morini\footnotemark[2] and E. Riccietti \footnotemark[3]}

\renewcommand{\thefootnote}{\fnsymbol{footnote}}
\footnotetext[2]{Dipartimento  di Ingegneria Industriale, Universit\`a di Firenze,
viale G.B. Morgagni 40,  50134 Firenze,  Italia,
stefania.bellavia@unifi.it, benedetta.morini@unifi.it}

\renewcommand{\thefootnote}{\fnsymbol{footnote}}
\footnotetext[2]{Dipartimento  di Matematica e Informatica ``Ulisse Dini'', Universit\`a di Firenze,
viale G.B. Morgagni 67a,  50134 Firenze,  Italia,
elisa.riccietti@unifi.it}

\maketitle{}

\begin{abstract}
In this paper we address the stable numerical solution of nonlinear ill-posed systems by a trust-region 
method. We show that an appropriate choice 
of the trust-region radius gives rise to a procedure that has the
potential to approach a solution of the unperturbed system. 
This regularizing property  is shown theoretically and validated numerically.
\end{abstract}


\section{Introduction}
Nonlinear systems modeling inverse problems are typically  ill-posed, in the sense that 
their solutions do not depend continuously on the data and their data are
affected by noise \cite{dsk, kns,vogel_book}.  In this work
we focus on  the stable  approximation of a solution
of  these problems. Procedures in the classes of Levenberg-Marquardt and trust-region
methods are discussed,  and  a  suitable version  of  
trust-region algorithm is shown to have regularizing properties both theoretically and numerically.
The underlying motivation for our study is twofold: most of the practical methods in literature  
have been designed for well-posed systems, see e.g., \cite{cgt, nw}, and thus are
unsuited in the context of  inverse problems; adaptation of existing procedures for handling ill-posed
problems,  carried out in the seminal papers \cite{gpp,hanke, hanke1,k,vogel,wy}, deserves further theoretical and numerical insights.
 
Let  
\begin{equation}\label{prob}
F(x)=y,
\end{equation}
with $F:\mathbb{R}^n\rightarrow \mathbb{R}^n$  continuously differentiable, be obtained from 
the discretization of a problem modeling an inverse problem. 
It is realistic to have  noisy data $y^\delta$ at disposal, satisfying
\begin{equation}\label{yd}
||y-y^\delta||_2\le \delta,
\end{equation}
for some positive  $\delta$.
Thus, in practice it is necessary to solve a problem  of the form 
\beqn{prob_noise}
F(x)=y^\delta,
\eeqn
and, due to ill-posedeness,  possible  solutions
may be arbitrarily far from those of the original problem.

In   \cite{hanke,hanke1}, Hanke supposed that 
an initial guess,   close enough to some  solution $x^\dagger$ of \req{prob}, is available.
Then, he proposed a {\em regularizing} Levenberg-Marquardt procedure which is able to compute 
a stable approximation $x_{k_*}^\delta$ to $x^\dagger$ or to some other solution of 
the unperturbed problem (\ref{prob})  close to $x^\dagger$. 
This task is achieved through a nonlinear stepsize control in the Levenberg-Marquardt procedure 
and the discrepancy principle as the stopping criterion, so that 
the iterative process is stopped at the iteration $k_*$ satisfying 
\begin{equation}\label{discrepance}
\|y^\delta-F(x_{k_*}^\delta)\|_2\le \tau \delta <\|y^\delta-F(x_{k}^\delta)\|_2,  \;\;\;0\le k< k_*,
\end{equation}
with $\tau>1$ appropriately chosen \cite{mor}.
Remarkably $x_{k_*}^\delta$ converges to a solution of (\ref{prob}) as $\delta$ tends to zero.

Further regularizing iterative methods have been proposed, including first-order methods and Newton-type methods.
Analogously to the Levenberg-Marquardt procedure proposed by Hanke,   instead of promoting  convergence to a solution of 
(\ref{prob_noise}), they form approximations
of increasing accuracy to  some solution of the unperturbed problem (\ref{prob}) 
until the discrepancy principle (\ref{discrepance}) is met. We refer to  \cite{dsk, kns} for the description and analysis of such methods.

The above mentioned regularizing  Levenberg-Marquardt method belongs to  the unifying framework of 
nonlinear stepsize control algorithms  for unconstrained optimization developed by Toint \cite{t}
and including trust-region methods  \cite{cgt}.
Therefore, elaborating on original ideas  by Hanke, we introduce and analyze a 
regularizing variant of the trust-region method. 
The main feature of our variant is the rule for selecting the trust-region radius which guarantees two properties.
First, it guarantees the same regularizing properties as the method by Hanke.
Second, as for standard trust-region procedure, it enforces a monotonic decrease  of
the value of  the function
\begin{equation}\label{mq}
 \Phi(x)=\frac{1}{2}\| y^\delta-F(x)\|_2^2,
\end{equation} 
at the iterates $\xkd$.
Convergence properties  are enhanced with respect to 
the regularizing Levenberg-Marquardt procedure in the following respects.
With exact data,  if  there exists an accumulation point  of the iterates
which solves  \req{prob},  then  any accumulation point of the sequence  solves \req{prob}.
With  noisy data, the methods have the potential to satisfy the discrepancy
principle \req{discrepance}. 
As for standard trust-region methods, these properties can be enhanced independently of the closeness of the initial guess 
to a solution of (\ref{prob}).

Our contribution  covers theoretical and practical aspects of the method proposed.
From a theoretical point of view, we propose the use of a trust-region radius
converging to zero as $\delta $ tends to zero. Trust-region methods with this distinguishing feature
have been proposed in several papers, see \cite{fan1, fan2,fan3, zw}, but none of such works 
was  either devised for ill-posed problems or applied to them; thus,
our study offers new insights on the potential of this choice for the trust-region radius.
Moreover, we have made a first attempt toward global convergent methods for 
ill-posed problems; to our knowledge, this topic has been considered only 
in a multilevel approach  proposed by  Kaltenbacher \cite{k}. 
Finally, local convergence analysis has been carried out  without the assumption (commonly made in literature) on 
the boundness of the inverse of the Jacobian $J$ of $F$, since it may not be fulfilled in the situation of ill-posedeness.
Taking into account that the standard convergence analysis of  trust-region  methods  always requires
the invertibility of $J$, our results  represent a progress in the theoretical investigation of convergence.
Concerning  numerical aspects,
we  discuss   an implementation of the regularizing trust-region method, 
and test its ability to approximate a solution of \req{prob}  in presence of noise.  Comparison 
with  a standard  trust-region scheme highlights the impact of the proposed trust-region radius choice on regularization.

The paper is organized as follows. In \S \ref{RLM} we describe the main features of  
the regularizing Levenberg-Marquardt method proposed by Hanke. 
In  \S \ref{defTR} we introduce our regularizing version  of the  trust-region methods 
and in \S \ref{local} we study the  local convergence properties. A comparative numerical analysis 
of all the procedures studied is done in \S \ref{exps}.

{\bf Notations.} We indicate the iterates of the procedures analyzed as $\xkd$; if the data are exact, $x_k$
may be used in alternative to $\xkd$. 
By $x_0^\delta=x_0$ we denote an initial guess which may incorporate a-priori knowledge of an exact solution.
The symbol $\|\cdot\|$ indicates the Euclidean norm.  The Jacobian matrix of $F$ is denoted as $J$.

\section{Regularizing Levenberg-Marquardt method  for ill-posed problems}\label{RLM}
We describe the regularizing version of the 
Levenberg-Marquardt method proposed  in \cite{hanke}
for solving (\ref{prob_noise}),
and analyze some issues for its practical implementation.

At  $k$-th iteration  of the Levenberg-Marquardt, 
given $\xkd\in \mathbb{R}^n$ and $\lambda_k>0$, let
\begin{equation}\label{lm}
\mklm(p)  =\frac{1}{2}\|\Fdk+\Jdk p\|^2+\frac{1}{2}\lambda_k \|p\|^2,
\end{equation}
be a quadratic model  around  $\xkd $ for  the function $\Phi$ in (\ref{mq}),  see \cite{l,m}.
The step  $p_k$ taken  minimizes $\mklm$, and  $x_{k+1}^\delta=\xkd+p_k$.
We observe that,   if $p(\lambda)$ is the solution of
\begin{equation}\label{system_1}
(B_k+\lambda I)p(\lambda) =-g_k,
\end{equation}
with $B_k=\Jdk^T\Jdk$ and $g_k=\Jdk^T(\Fdk)$, then  $p_k=p(\lambda_k)$.

If problem \req{prob_noise} is  ill-posed,
and  the scalars $\lambda_k$  are limited to promote convergence of  procedure, see 
\cite{m},  then  the solution of   \req{prob} may be significantly  misinterpreted 
\cite{g,kns, vogel_book}.  The regularizing Levenberg-Marquardt method \cite{hanke}  
attempts to approximate solutions of \req{prob}  by choosing
$\lambda_k$  as  the solution  $\lambda_k^q$  of the nonlinear scalar equation
\begin{equation} \label{seculare_q}
 \|\Fdk +\Jdk p(\lambda)\|= q \|\Fdk\|,
\end{equation}
for some fixed $q\in (0, 1) $.
Under suitable assumptions   discussed below,  $\lambda_k^q$ is uniquely
determined from (\ref{seculare_q}).

As for (\ref{seculare_q}), 
it is useful to establish relations between  
$\lambda$, $\|p(\lambda)\|$ and  $\|\Fdk+ \Jdk p(\lambda) \|$.
\vskip 5pt
\begin{lemma}\label{Sigma_decomp} \cite[Lemma 4.2]{bcgmt}
Suppose $\|g_k\|\neq 0$ and let $p(\lambda)$ be the minimum norm solution of 
\req{system_1} with  $\lambda \geq 0$.
Suppose furthermore that $\Jdk$ is of rank $\ell$ and its singular-value
decomposition  is given by $U_k \Sigma_k V_k^T$ where
$\Sigma_k$ is the diagonal matrix with entries  $\varsigma_1,\ldots,\varsigma_{n}$ on the diagonal. 
Then, denoting $r=(r_1,r_2,\dots,r_n)^T= U_k^T(\Fdk)$, we have that 
\begin{eqnarray}
& & \|p(\lambda)\|^2 
 =  \sum_{i=1}^\ell \frac{\varsigma_i^2 r_i^2}{(\varsigma_i^2 + \lambda)^2} ,
\label{Sigmanorms1}  \\
& & \|\Fdk +\Jdk p(\lambda) \|^2
=  \sum_{i=1}^\ell \frac{\lambda^2 r_i^2}{(\varsigma_i^2 + \lambda)^2}
  +\sum_{i=\ell+1}^n r_i^2.\label{Sigmanorms2} 
\end{eqnarray}
\end{lemma}
\vskip 5pt \noindent
Using this result,  the solution of (\ref{seculare_q}) is characterized as follows. 

\vskip 5pt
\begin{lemma}\label{sol_cond_q}
Suppose $\|g_k\|\neq 0$. Let $p(\lambda)$ be the minimum norm solution of 
\req{system_1} with  $\lambda \geq 0$,  $\mathcal{R}(\Jdk)^\bot$ be 
the orthogonal complement of  the range  $\mathcal{R}(\Jdk)$ of $\Jdk$,
and $P_k^\delta$ be the orthogonal projector onto $\mathcal{R}(\Jdk)^\bot$. 
Then
\begin{description}
\item{(i)}  Equation \req{seculare_q} is  not solvable if $\|P_k^\delta (\Fdk)\| > q\|\Fdk\|$.
\item{(ii)} If 
\begin{equation}\label{gamma_libro}
\|\Fdk+\Jdk(x^\dagger-\xkd)\|\le \frac{q}{\theta_k}\|\Fdk\|,
\end{equation}
for some $\theta_k>1$, then equation \req{seculare_q} has a unique solution  
$ \lambda_k^q $ such that
\begin{equation}\label{bound_lambda}
\lambda_{k}^{q}\in  \left( 0, \frac{q}{1-q}\|B_k\| \right].
\end{equation}
\end{description}
\end{lemma}
\begin{proof}
$(i)$ Equation (\ref{Sigmanorms2})  implies 
\begin{eqnarray*}
& & \lim_{\lambda \rightarrow 0} \|\Fdk +\Jdk p(\lambda) \|=\|P_k^\delta (\Fdk)\| , \\
& & \lim_{\lambda \rightarrow \infty} \|\Fdk +\Jdk p(\lambda) \|=\|\Fdk\|.
\end{eqnarray*}
Thus, since  $\|\Fdk +\Jdk p(\lambda) \|$ is  
monotonically increasing as a function of $\lambda$, 
we conclude that \req{seculare_q} does not admit solution if  
$\|P_k^\delta (\Fdk)\| > q\|\Fdk\|$.

$(ii)$ Trivially  $\|P_k^\delta (\Fdk)\|\le \|\Fdk+\Jdk(x-\xkd)\| $, for any $x$.
Hence,  for the monotonicity of $\|\Fdk +\Jdk p(\lambda) \|$,
if  \req{gamma_libro} holds, then equation  \req{seculare_q} admits a solution which
is positive and unique.  Finally, 
$$
\Fdk +\Jdk p(\lambda)=\lambda(\Jdk\Jdk^T+\lambda I)^{-1}(\Fdk),
$$
see e.g., \cite[Proposition 2.1]{hanke},  and by  \req{seculare_q}
\begin{eqnarray*}
q \|\Fdk\|
&=& \lambda_{k}^{q}\|(\Jdk\Jdk^T+\lambda_{k}^{q} I)^{-1}(\Fdk)\|\\
&\ge& \frac{\lambda_{k}^{q} }{\|\Jdk\|^2+\lambda_{k}^{q}} \|\Fdk\|,
\end{eqnarray*}
which yields \req{bound_lambda}.
\end{proof} 
\vskip 5pt

In \cite{hanke}, the   analysis of  resulting  Levenberg-Marquardt method   was made   under 
the subsequent assumptions  on the solvability of the problem \req{prob},
the Taylor remainder of  $F$, and the vicinity of
the initial guess $x_0$ to some  
solution $x^\dagger$ of \req{prob}. 
 \begin{ipotesi}\label{ALM} 
Given an initial guess $x_0$, there exist  positive $\rho$  and $c$ such that  
system (\ref{prob}) is solvable in  $ B_{\rho} (x_0)$, and 
\begin{equation} \label{condfond}
\quad  \|F(x)-F(\tilde x)-J(x)(x-\tilde x)\|\le c \|x-\tilde x\|\, \|F(x)-F(\tilde x)\|, \;\;\; x,\tilde x \in B_{2\rho} (x_0).
\end{equation}
\end{ipotesi}
\begin{ipotesi} \label{A5} Let $x_0$,  $c$ and  $\rho$ as in Assumption \ref{ALM}, 
$x^\dagger$ be a solution of \req{prob} and $x_0$ satisfy 
\begin{eqnarray}
\|x_0-x^\dagger\|&<& \min\left\{\frac{q}{c},  \rho\right \},   \hspace*{43pt}\mbox{if}\;\;\;\;\delta=0,\label{loc1}\\
\|x_0-x^\dagger\|& <&  \min\left\{\frac{q\tau-1}{c(1+\tau)},  \rho \right\}, \;\;\;\;\mbox{if}\;\;\;\;\delta>0 \label{loc2}
\end{eqnarray}
where $\tau>1/q$.
\end{ipotesi}
\vskip5pt
From Assumption \ref{ALM} it follows  that 
inequality (\ref{gamma_libro}) is satisfied for any $x_k^\delta$ belonging to  $ B_{2\rho} (x_0)$ 
and consequently there  exists a solution to (\ref{seculare_q}),  
see \cite[Theorems 2.2, 2.3]{hanke}.

 Under Assumptions \ref{ALM} and \ref{A5},  the approximations $x_{k^*}^\delta$ 
generated by the Levenberg-Marquardt method satisfy (\ref{discrepance}) and converge
to a solution of (\ref{prob}) as $\delta$ tends to zero.
\vskip 5pt
\begin{theorem}\label{conv_lm} 
Let Assumptions  \ref{ALM} and \ref{A5} hold and $ \xkd $ be the 
Levenberg-Marquardt  iterates  determined by using  \req{seculare_q}. 
For noisy data, suppose 
$k<k_*$  where $k_*$ is defined in \req{discrepance}. 
Then,   any iterate $x_k^\delta$ belongs to $B_{2\rho}(x_0)$. 
With exact data,   the sequence $\{x_k\}$ converges to a solution of \req{prob}.
With noisy data, the stopping criterion \req{discrepance} is satisfied after a finite number $k_*$ of iterations
and $\{x_{k^*}^\delta\}$ converges to a solution of \req{prob} as $\delta$ goes to zero.
\end{theorem}
\begin{proof}
See  \cite{hanke}, Theorem 2.2 and Theorem 2.3.
\end{proof}
\vskip 5pt

Let us focus on a specific issue concerning the implementation of the method which, 
to our knowledge,   has not been addressed either in \cite{hanke} or in related papers.
The   numerical solution of (\ref{seculare_q}) requires the application of a root-finder method
and Newton method  is the most efficient procedure,   though in general it requires
the knowledge of an accurate approximation to the solution. On the other hand,
nonlinear equations which are  monotone and convex (or concave) 
on some interval  containing the root are particularly suited to an application of Newton method,
see e.g. \cite[Theorem  4.8]{henrici}. Equation (\ref{seculare_q}) does not have such properties
but it can be replaced  by an equivalent equation  with   strictly decreasing  and concave function in $[\lambda_k^q, \infty)$; 
thus, Newton method applied to the reformulated equation converges globally to
$\lambda_k^q$ whenever the initial guess overestimates such a root.
\vskip 5pt
\begin{lemma}\label{sec_cond_q}
Suppose $\|\Fdk\|\neq 0$, and that (\ref{seculare_q}) has   positive solution $\lambda_k^q$. Let
\begin{equation}\label{sec_new}
\psi(\lambda)=\frac{\lambda}{\|\Fdk +\Jdk p(\lambda) \|}-\frac{\lambda}{q\|\Fdk\|} =0.
\end{equation}
Then, Newton method applied to \req{sec_new}
converges monotonically and globally to the  root $\lambda_k^q$ of (\ref{seculare_q}) 
for any initial guess in the interval $[\lambda_k^q, \infty)$.
\end{lemma}
\begin{proof}
Trivially, solving   (\ref{seculare_q}) is equivalent to finding the  positive  root of the equation
(\ref{sec_new}). We  now show that  $\psi(\lambda)$ is  strictly decreasing in $[\lambda_k^q, \infty)$ and concave on $(0, \infty)$. By (\ref{Sigmanorms2}),  
\begin{equation}\label{rapl}
\frac{\lambda}{\|\Fdk +\Jdk p(\lambda) \|}= \left (\, \sqrt{\sum_{i=1}^l \left(\frac{r_i}{\zeta_i^2+\lambda}\right)^2+
\sum_{i=l+1}^n \left(\frac{r_i}{\lambda}\right)^2} \, \right)^{-1},
\end{equation}
and this function is concave on $(0, \infty)$, cfr. \cite[Lemma 2.1]{cgt_ima}. Thus,  $\psi$ is concave   on $(0, \infty)$
and trivially $\psi'(\lambda)$ is strictly decreasing. 

Now we show that $\psi'(\lambda_k^q)$ is negative; thus, using the monotonicity of $\psi'(\lambda)$, we 
get  that  $\psi(\lambda)$ is strictly decreasing in $[\lambda_k^q, \infty)$. Differentiation of $\psi(\lambda)$ and 
(\ref{seculare_q}) give 
\begin{eqnarray*}
\psi'(\lambda_k^q) &=&  \frac{(\lambda_k^q)^3}{\|\Fdk +\Jdk p(\lambda_k^q) \|^3}  
\left(\sum_{i=1}^l \frac{r_i^2}{(\zeta_i^2+\lambda_k^q)^3}+ \sum_{i=l+1}^n \frac{r_i^2}{(\lambda_k^q)^3} \right)
\, -\, \frac{1}{q\|\Fdk\|}
\\ &=& \frac{(\lambda_k^q)^2}{\|\Fdk +\Jdk p(\lambda_k^q) \|^3} 
\left (\,  \sum_{i=1}^l \frac{r_i^2\lambda_k^q}{(\zeta_i^2+\lambda_k^q)^3}+ \sum_{i=l+1}^n \left(\frac{r_i}{\lambda_k^q} \right)^2  
-\frac{\|\Fdk +\Jdk p(\lambda_k^q) \|^2}{(\lambda_k^q)^2}\, \right).
\end{eqnarray*}
Moreover, using (\ref{rapl}), it holds
\begin{eqnarray*}
\psi'(\lambda_k^q) &=& \frac{(\lambda_k^q)^2}{\|\Fdk +\Jdk p(\lambda_k^q) \|^3} 
\left (\,   \sum_{i=1}^l \frac{r_i^2\lambda_k^q}{(\zeta_i^2+\lambda_k^q)^3}
- \sum_{i=1}^l \left(\frac{r_i}{\zeta_i^2+\lambda_k^q}\right)^2 \, \right) \\ 
&=& -
\frac{(\lambda_k^q)^2}{\|\Fdk +\Jdk p(\lambda_k^q) \|^3} 
  \sum_{i=1}^l \frac{r_i^2\zeta_i^2}{(\zeta_i^2+\lambda_k^q)^3}, \\
\end{eqnarray*}
i.e. $\psi'(\lambda_k^q)$ is negative. 

The  claimed convergence  of Newton method 
follows from  results  on univariate concave functions  given in \cite[Theorem  4.8]{henrici}.
\end{proof} 
\vskip 5pt
For the practical evaluation of $\psi(\lambda)$ and $\psi'(\lambda)$ we refer to \cite{cgt,more}.

Since  (\ref{seculare_q}) may have not solution,
Hanke  observed that  it may be replaced by 
\begin{equation}\label{RQ}
 \|F(\xkd)-y^\delta +J(\xkd) p_k\| \ge   q \|F(\xkd)-y^{\delta}\|,
\end{equation}
later denoted as the {\em  q-condition}, \cite[Remark p. 6]{hanke} but 
this criterion  was not analyzed or employed in numerical computation.
Since our aim is to tune $\lambda_k$ in view of global convergence, while   
preserving regularizing properties, in the next section  we allow more flexibility in  
its selection and  design a trust-region method based on  condition (\ref{RQ}).
 
\section{A regularizing trust-region method}\label{defTR}
Trust-region methods are globally convergent approaches where 
the stepsize between two successive iterates is determined via a
nonlinear control mechanism \cite{cgt}.
At a generic iteration $k$ of a trust-region method,  the step 
$p_k$  solves
\begin{equation}\label{TR}
\begin{array}{l}
\displaystyle \min_p \mktr(p)=\frac{1}{2}\|\Fdk+\Jdk p\|^2,\\
\mbox { s.t. } \|p\|\le \Delta_k,  
\end{array}
\end{equation}
where $\Delta_k$ is a given positive trust-region radius.
If $\|g_k\|\neq 0$ then $p_k$ solves \req{TR} if and only if 
it satisfies \req{system_1} for some nonnegative  $\lambda_k$  such that 
\begin{equation}\label{sol_tr}
\lambda_k(\|p_k\|-\Delta_k)=0.
\end{equation}
Therefore, whenever the minimum norm solution $p^+$ of 
$$
B_kp=-g_k,
$$
satisfies  $\|p^+\|\le \Delta_k$, then  the scalar 
$\lambda_k$ is null and $p_k=p(0)$ solves   \req{TR}. 
Otherwise,  the step $p_k=p(\lambda_k)$ is a Levenberg-Marquardt step. 
If $\|p_k\|=\Delta_k$, then the  trust-region is said to be active.

Starting from an arbitrary initial guess, 
trust-region methods generate a sequence of iterates such 
that the value of $\Phi$ in (\ref{mq})  is monotonically decreasing and 
this feature is enforced by an adaptive  choice of  the  radius $\Delta_k$.
Specifically,  let $p_k$ be the trust-region step and 
\begin{equation}\label{rho_tr}
\rho_k=\frac{ared(p_k)}{pred(p_k)},
\end{equation}
be  the ratio between the  achieved $ared(p_k)$ and predicted $pred(p_k$) reductions given by
\begin{eqnarray}
& & ared(p_k)=  \Phi(\xkd)-\Phi(\xkd+p_k), \label{ared}\\
& &  pred(p_k)= \Phi(\xkd)-\mktr(p_k) . \label{pred}
\end{eqnarray}
Then, the trust region radius is reduced if $\rho_k$ is below some small positive
threshold; otherwise it is left unchanged or enlarged \cite{cgt}.

Since trust-region steps  and Levenberg-Marquardt steps have the same
form  \req{system_1}, trust-region and Levenberg-Marquardt methods
fall into a single unifying framework which can be 
used for their   description and theoretical analysis 
\cite{cgt_ima,  more, t}. 
Due to such a strict connection, we elaborate on the regularizing   Levenberg-Marquardt 
described in the previous section,
and introduce a regularizing variant 
of trust-region  methods for  solving ill-posed problems.

The standard trust-region strategy is modified so that the nonlinear stepsize
control  enforces both the monotonic reduction of $\Phi$ and the  $q$-condition (\ref{RQ}).
To this end,  we first  characterize the parameters $\lambda$ such that $p(\lambda)$ satisfies (\ref{RQ}).
\vskip 5pt
\begin{lemma}\label{sol_q}
Assume $\|g_k\|\neq 0$. Let $p(\lambda)$ be the minimum norm solution of 
\req{system_1} with  $\lambda \geq 0$ and $P_k^\delta$ be the orthogonal projector onto $\mathcal{R}(\Jdk)^\bot$.
Then, equation \req{RQ} is satisfied for any $\lambda\ge 0$ 
whenever  
\begin{equation}\label{cond_lambda}
\|P_k^\delta (\Fdk)\| \ge  q\|\Fdk\|.
\end{equation} 
Otherwise, it 
is satisfied for any $\lambda\ge \lambda_k^q$ where $\lambda_k^q$ satisfies \req{bound_lambda}. 
\end{lemma}
\begin{proof}
The claims easily follow from   Lemma \ref{sol_cond_q}.
\end{proof} 
\vskip 5pt
Now we are ready to characterize the size of the trust-region radius guaranteeing (\ref{RQ}). 
\vskip 5pt
\begin{lemma}\label{delta_qcond}
Let $p_k$ solve the trust-region problem (\ref{TR}). If 
\begin{equation}\label{radius}
\Delta_k\le \frac{1-q}{\|B_k\|}\|g_k\|,
\end{equation}
then $p_k$ satisfies the $q$-condition (\ref{RQ}).
\end{lemma}
\begin{proof}
By Lemma \ref{sol_q} we know that 
the $q$-condition  is satisfied either for $\lambda\ge 0$, or 
for  any $\lambda\ge \lambda_k^q$.
In the former case, the claim trivially holds. In the latter case, by
\req{system_1}   it follows
$$
\|p(\lambda_k^q)\| \ge  \frac{\|g_k\|}{\|B_k+\lambda_k^qI\|},
$$
and  by \req{bound_lambda} it holds
$$
\|B_k+\lambda_k^qI\| \le \frac{\|B_k\|}{1-q}.
$$
 By construction  $\|p_k\|\le \Delta_k$, and  if (\ref{radius}) holds then we obtain
$$
\|p_k\|=\|p(\lambda_k)\|\le \frac{1-q}{\|B_k\|}\|g_k\| \le  \frac{\|g_k\|}{\|B_k+\lambda_k^qI\|} \le  
\|p(\lambda_k^q)\|.
$$
Since $\|p(\lambda)\|$ is monotonically decreasing, it follows
$\lambda_k\ge\lambda_k^q$ and condition \req{RQ} is satisfied. 
\end{proof}
\vskip 5pt
We stress that the bound \req{radius} provides a practical rule for choosing the trust-region radius that guarantees the satisfaction of  
the q-condition \req{RQ}. Conversely,  in papers  \cite{wy} and \cite{zw}, where  trust-region  for ill-posed problems are studied,
such a condition is respectively assumed to be satisfied and  explicitely enforced rejecting the step whenever it does not hold.

The   result in Lemma \ref{delta_qcond} suggests the trust-region iteration described in Algorithm \ref{algoTR}.
We distinguish between the parameters 
needed in the case of exact-data and the parameters required with noisy-data.

\algo{TRalgo}{$k$th iteration of the regularizing Trust-Region method for problem  \req{prob_noise}}{
Given $\xkd$,    $\eta\in (0,1)$, $\gamma\in (0,1)$, $0<C_{\min}<C_{\max}$. 
\vskip 1pt
Exact-data:  given $\delta=0$, $q\in (0,1)$.
\vskip 1pt
Noisy-data: given $\delta>0$, $\tau>1$, $q>1/\tau$.
\\[1ex]
\begin{description}
\item[1.]  Compute $B_k=\Jdk^T\Jdk$ and $g_k=\Jdk^T(\Fdk)$.
\item[2.]  Choose $\displaystyle\Delta_k\in \left.\left [C_{\min}\|g_k\|, 
\min\left\{ C_{\max},  \frac{1-q}{\|B_k\|}  \right\}\|g_k\|\right]\right.$.
\item[3.]  Repeat\\ 
 3.1 Compute the solution $p_k$ of the trust-region  problem (\ref{TR}).
 \\
 3.2  Compute $\rho_k$ given in \req{rho_tr}--\req{pred}.
\\
3.3 If $\rho_k< \eta$, then set $\Delta_{k}=\gamma \Delta_k$.\\
 \hspace*{-25pt} Until $\rho_k\ge \eta$.
\item [4.] Set $\xk+1d=\xkd+p_k$.
\end{description}
}\label{algoTR}

Due to well know properties of trust-region methods,   
Algorithms \ref{algoTR} is  well-defined, i.e. if $\|g_k\|\ne 0$ 
the step $p_k$ is found  within a finite number of attempts, provided that the following Assumption is met \cite{cgt}.  
\vskip 5pt
\begin{ipotesi}\label{A2} 
There exists a positive constant $\kappa_J$ such that
$$
\|J(x)\|\le \kappa_J,
$$
 for any  $x$ belonging to the level set 
${\cal L}=\{ x\in \mathbb{R}^n \, \mbox{ s.t. }\, \Phi(x)\le \Phi(x_0)\}$.
\end{ipotesi}
\vskip 5pt

Global convergence  of the trust-region method is stated in the following theorem;
we refer to \cite[Theorem 11.9]{nw}  for the proof. 
\begin{theorem} 
Suppose that Assumption \ref{A2} holds and $J$ is Lipschitz continuous on $\mathbb{R}^n$. 
Then, the sequence $\{\xkd\}$ generated by  Algorithm \ref{algoTR} satisfies
\begin{equation}\label{conv_grad}
\lim_{k\rightarrow \infty} \nabla\Phi(\xkd)=\lim_{k\rightarrow \infty}\|\Jdk^T(\Fdk)\|=0.
\end{equation} 
\end{theorem}

We observe that assumption on   Lipschitz continuity of $J$ is made  also in the paper \cite{k}.
By construction, the  sequence $ \{\|\Fdk\|\}$ is monotonically decreasing
and bounded below by zero; hence it is convergent. 
Equation \req{conv_grad} implies that any accumulation point of the sequence $\{\xkd\}$ 
is a stationary point  of $\Phi$. 
As for exact data, we conclude that  if there exists an accumulation point of $\{x_k\}$ solving \req{prob}, 
then any  accumulation point of the sequence solves \req{prob}. 
In the case of noisy data, if the value of $\Phi$ at some accumulation point of  $\{\xkd\}$ 
is below the scalar $\tau \delta$, then there exists an iterate   $x_{k _*}^\delta$ 
such that the discrepancy principle is met.

It remains to show the behaviour of the  sequences generated
by   Algorithm \ref{algoTR} when, for some $k$,  
$\xkd$ is sufficiently close to a solution $x^\dagger$ of \req{prob}.
For instance,
this occurs with exact data when 
the accumulation points of  $\{x_k\}$ solve \req{prob} 
and $k$ is sufficiently large. In the next section we show that
the trust-region method described in  Algorithm \ref{algoTR} shares
the same  local regularizing properties as the regularizing Levenberg-Marquardt method.


\section{Local behaviour of the trust-region method}\label{local}
We analyze the local properties of the  trust-region method 
under the same assumptions made for the Levenberg-Marquardt method.
Hence, we suppose that there exists an iteration index $\bar k$ such that the iterate $x_{\bar k}^\delta$ satisfies the following assumptions
that are the counterpart of  Assumptions \ref{ALM} and \ref{A5} for the Levenberg-Marquardt method.

\begin{ipotesi}\label{ATR} 
 Suppose that for  some iteration index $\bar k$
there exist  positive $\rho$  and $c$ such that  
system (\ref{prob}) is solvable in  $ B_{\rho} (x_{\bar k}^\delta)$, and 
\begin{equation} \label{condfond_TR}
\quad  \|F(x)-F(\tilde x)-J(x)(x-\tilde x)\|\le c \|x-\tilde x\|\, \|F(x)-F(\tilde x)\|, \;\;\; x,\tilde x \in B_{2\rho} (x_{\bar k}^\delta).
\end{equation}
Moreover,  letting  $x^\dagger$ be a solution of \req{prob}, and  $\tau>1/q$ if the data are noisy,  
suppose that $x_{\bar k}^\delta$  satisfies 
\begin{eqnarray}
\|x_{\bar k}-x^\dagger\|&<& \min\left\{\frac{q}{c},  \rho\right \},   \hspace*{43pt}\mbox{if}\;\;\;\;\delta=0,\label{locTR1}\\
\|x_{\bar k}^\delta-x^\dagger\|& <&  \min\left\{\frac{q\tau-1}{c(1+\tau)},  \rho \right\}, \;\;\;\;\mbox{if}\;\;\;\;\delta>0\label{locTR2}.
\end{eqnarray}
\end{ipotesi}

To our knowledge,  except for paper \cite{wy,zw},
local convergence properties of trust-region  strategies 
have been  analyzed under assumptions 
which involve the inverse of $J$ and its upper bound in a neighbourhood of a solution,
and thus are stronger than \req{condfond_TR}.

The following theorems  show the local behaviour of the regularizing trust-region method. 
We prove that  locally the trust-region is active, the iterates $x_k^\delta$ with $k>\bar k$ remain into the  ball  $B_{\rho} (x_{\bar k}^\delta)$ and the resulting algorithm is regularizing. We remark that in standard trust-region methods,
the trust-region becomes eventually inactive. On the other hand, regularization requires strictly positive scalars $\lambda_k$,
and consequently an active trust-region in all iterations.
First we focus on the noise-free case and we show that the error 
$\| x_{k}- x^\dagger\|$ decreases in a monotonic way for $k\ge\bar k$, and the sequence 
$\{x_{k}\}$ converges to a solution of (\ref{prob}). 
 \vskip 5pt

\begin{theorem}\label{lambdaTR_bound}
Suppose that  Assumptions \ref{A2} and \ref{ATR}  hold and $\delta=0$. Then,  Algorithm \ref{algoTR} generates a sequence 
$\{x_{k}\}$  such that, for $k>\bar k$,
\begin{description}
\item{(i)}   the trust-region is active, i.e. $\lambda_k>0$ and 
$x_k$ belongs to $B_{\rho}(x_{\bar k})$;
\item{(ii)} $\| x_{k+1}- x^\dagger\|<\| x_{k}- x^\dagger\|$;
\item{(iii)} there exists a constant $\bar \lambda > 0$ such that $ \lambda_k\le \bar \lambda$.
\end{description}
Moreover, 
\begin{description}
\item{(iv)} the sequence $\{x_k\}$  converges to a solution of \req{prob}.
\end{description}
\end{theorem}
\begin{proof} 
{\it (i)-(ii)}  The scalar $\lambda_{\bar k}$ in Algorithm \ref{algoTR} is  such that 
$\lambda_{\bar k}\ge \lambda_{\bar k}^q$.
From  Assumption \ref{ATR},     condition 
\req{gamma_libro} is satisfied at 
$k=\bar k$ with $\theta_{\bar k}= \displaystyle \frac{q}{c\|x_{\bar k}-x^\dagger\|}$, 
and consequently  by Lemma \ref{sol_cond_q}  $\lambda_{\bar k}^q$ is strictly positive.
Hence, the trust-region is active.
Further,  by  a straightforward adaptation of   \cite[Proposition 4.1]{kns}\footnote{cfr.  equation  (4.6)  in  \cite[Proposition 4.1]{kns}} 
it follows
$$ 
\|x_{\bar k}-x^\dagger\|^2-\|x_{\bar k+1}-x^\dagger\|^2\ge  \frac{2(\theta_{\bar k}-1)}{\theta_{\bar k} \lambda_{\bar k}}\|F(x_{\bar k})-y+J(x_{\bar k}) p_{\bar k}\|^2,
$$
and this implies 
$\|x_{\bar k+1}-x^\dagger\|<\|x_{\bar k}-x^\dagger\|$ as  (\ref{locTR1}) guarantees  $\theta_{\bar k}>1$. Consequently $x_{\bar k+1}\in B_{\rho} (x_{\bar k})$.  
Repeating the above  arguments,  by induction we can prove that condition \req{gamma_libro} holds     for $k\ge \bar k$, $\lambda_k>0$, and the following inequality holds
\begin{equation}\label{monoton_1_new}
\qquad \ 
\|x_k-x^\dagger\|^2-\|x_{k+1}-x^\dagger\|^2\ge  \frac{2(\theta_k-1)}{\theta_k \lambda_k}\|F(x_{k})-y+J(x_{ k}) p_k\|^2,
\end{equation}
with
$$\theta_k=\displaystyle\frac{q}{c\|x^\dagger-x_k\|}>1.$$
Thus, by induction, the sequence $\{\|x_k-x^\dagger\|\}_{k=\bar k}^\infty$ is monotonic decreasing  and    the sequence $\{\theta_k\}_{k=\bar k}^\infty$  is monotonic increasing.
 
{\it (iii)} Since  the trust-region is active, by \req{system_1}
\begin{equation}\label{bound_supl}
\Delta_k=\|p_k\|=\|(B_k+\lambda_kI)^{-1}g_k\|\le \frac{\|g_k\|}{\lambda_k}.
\end{equation}
Thus  our claim follows if  $\Delta_k/\|g_k\|$ is larger than a suitable threshold, independent from $k$.
Let us provide such a bound by estimating the value of $\Delta_k$
which guarantees condition  $\rho_k\ge\eta$. 
If this condition  is fulfilled for the value of $\Delta_k$
fixed in Step 2 of Algorithm \ref{algoTR}, then $\Delta_k/\|g_k\|\ge C_{\min}$;
otherwise, the trust-region radius is progressively reduced, and we provide a bound
for the value of $\Delta_k$ at termination of Step 3 of  Algorithm \ref{algoTR}
in the case where  $\Phi(x_k+p_k)>\mktr(p_k)$. This occurrence represents the 
most adverse case; in fact if  $\Phi(x_k+p_k)\le\mktr(p_k)$
then $\rho_k\ge 1>\eta$ and the repeat loop  terminate for a trust-readius greater than 
or equal to the one estimated below.
Trivially, 
$$
1-\rho_k = \frac{\Phi(x_k+p_k)-\mktr(p_k)}{\Phi(x_k)-\mktr(p_k)},
$$
and 
\begin{eqnarray*}
\Phi(x_k+p_k)-\mktr(p_k)&\le & \frac{1}{2}\|F(x_k+p_k)-F(x_k)-J(x_k) p_k\|^2   \\
&&+\|F(x_k+p_k)-F(x_k)-J(x_k) p_k\|\|F(x_k)-y+J(x_k) p_k\| .
\end{eqnarray*}
By \req{condfond_TR} and the mean value \cite[Theorem 11.1]{nw}, it holds
\begin{equation}\label{ineq_p2} \ \ \
\|F(x_k+p_k)-F(x_k)-J(x_k) p_k\|\le c \|p_k\| \|F(x_k+p_k)-F(x_k)\|\le c\kappa_J \|p_k\|^2.
\end{equation}
Consequently, as $\Delta_k\le C_{\max}\|g_k\|$, 
$$
\Phi(x_k+p_k)-\mktr(p_k)  \le \frac{1}{2}  c\kappa_J \Delta_k^2\|F(x_{0})-y\| (c\kappa_J^3C_{\max}^2\|F(x_{0})-y\|+2).
$$
Theorem  6.3.1 in \cite{cgt} shows that 
$$
\Phi(x_k) -\mktr(p_k)
\geq  \frac{1}{2}\|g_k\|
\min \left\{ \Delta_k,\frac{\|g_k\|}{\|B_k\|} \right\}.
$$
Then, 
$$
\Phi(x_k)-\mktr(p_k)\geq \frac{1}{2}\Delta_k \|g_k\|,
$$
whenever  $\Delta_k\le \displaystyle \frac{\|g_k\|}{\kappa_J^2}$
and this implies
$$
1-\rho_k\le \frac{c\kappa_J\Delta_k \|F(x_{0})-y\| (c \kappa_J^3C_{\max}^2\|F(x_{0})-y\|+2)}{  \|g_k\|}.
$$
Namely, termination of the repeat loop occurs with
$$
 \Delta_k\le \|g_k\| \omega,\nonumber
$$
and
\begin{equation}\label{omega}
\omega=\min \left\{\frac{1}{\kappa_J^2},\frac{1-\eta}{c\kappa_J\|F(x_{0})-y\|  (c  \kappa_J^3C_{\max}^2\|F(x_{0})-y\|+2)} \right\}.
\end{equation}
Taking into account Step 2 and  the updating rule at Step 3.3, we can conclude that,
at termination of Step 3, the trust-region radius $\Delta_k$ satisfies
$$
\Delta_k\ge \min\left \{C_{\min} , \, \gamma \omega
 \right \}\|g_k\| .
$$
Finally, by \req{bound_supl} $\lambda_k\le \bar \lambda$ as   
\begin{equation} \label{bound_lambda_1}
\lambda_k\le \frac{\|g_k\|}{\Delta_k} \le \max \left \{  \frac{1}{ \gamma  \omega},\, \frac{1}{C_{\min}}
\right \}.
\end{equation}

{\it (iv)} Since both the function $(\theta-1)/\theta$ and the sequence $\{\theta_k\}_{k=\bar k}^\infty$
are monotonic increasing,  it follows 
$(\theta_k-1)/\theta_k>(\theta_{\bar k}-1)/\theta_{\bar k}$ and by \req{monoton_1_new} 
\begin{equation}\label{monoton_3}
\|x_k-x^\dagger\|^2-\|x_{k+1}-x^\dagger\|^2\ge\frac{2(\theta_{\bar k}-1)}{\theta_{\bar k} \lambda_k}\|F(x_k)-y+J(x_k) p_k\|^2.
\end{equation}
Therefore, following the lines of the proof of  Theorem 4.2 in \cite{kns},
we can conclude that $\{x_k\}$ is a Cauchy sequence, i.e., it is convergent.

Finally,  by  \req{monoton_3},   $\lambda_k\le \bar \lambda$ 
and \req{RQ}  
$$
\|x_k-x^\dagger\|^2-\|x_{k+1}-x^\dagger\|^2\ge\frac{2(\theta_{\bar k}-1)q^2}{\theta_{\bar k}  \bar \lambda} \|F(x_k)-y\|^2,
$$
Hence  $ \|F(x_k)-y\|$ tends to zero and the limit of $x_k$ has to be a solution of (\ref{prob}). 

\end{proof}

A similar  result can be given for the noisy case. In the following theorem we prove that for $\bar k<k<k_*$, where
  $k_*$ is defined in \req{discrepance}, the trust region is active and therefore $\lambda_k>0$.
Moreover, the stopping criterion 
\req{discrepance} is satisfied after a finite number  of iterations and the method is regularizing as  the error decreases monotonically and 
 the sequence  $\{x_{k_*}^\delta\}$ converges to a solution of \req{prob} whenever $\delta$ goes to zero.

\begin{theorem}\label{lambdaTR_bound_noise}
Suppose that  Assumptions \ref{A2} and \ref{ATR}  hold and $\delta>0$. Moreover, 
 suppose  $\tau>1/q$ and  $\bar k<k_*$, where $\tau$ and $k_*$ are defined in \req{discrepance}.
Then,  Algorithm \ref{algoTR} generates a sequence  $\xkd$    such that, for $\bar k\le k<k_*$,
\begin{description}
\item{(i)} the trust-region is active, i.e. $\lambda_k>0$ and 
$\xkd$ belongs to $B_{\rho}(x_{\bar k}^\delta)$; 
\item{(ii)} $\| x_{k+1}^\delta- x^\dagger\|<\| \xkd- x^\dagger\|$;
\item{(iii)} there exists a constant $\bar \lambda > 0$ such that $ \lambda_k\le \bar \lambda$.
\end{description}
Moreover,
\begin{description}
\item{(iv)} the stopping criterion 
\req{discrepance} is satisfied after a finite number $k_*$ of iterations and  
 the sequence  $\{x_{k_*}^\delta\}$ converges to a solution of \req{prob} whenever $\delta$ goes to zero.
\end{description}
\end{theorem}
\begin{proof} 
{\it (i)-(ii)} By \req{condfond_TR}  and \req{yd} we get
\begin{eqnarray*}
\|y^\delta- F(x_{\bar k}^\delta)-J(x_{\bar k}^\delta)(x^\dagger-x_{\bar k}^\delta)\|&\le& \delta+ 
\|y -F(x_{\bar k}^\delta)-J(x_{\bar k}^\delta)(x^\dagger-x_{\bar k}^\delta)\| \\
& \le& \delta+c\|x^\dagger-x_{\bar k}^\delta\|\, \|y-F(x_{\bar k}^\delta)\| \\
&\le & (1+ c\|x^\dagger-x_{\bar k}^\delta\|)\delta+ c\|x^\dagger-x_{\bar k}^\delta\|\, \|y^\delta-F(x_{\bar k}^\delta)\| .
\end{eqnarray*}
Then, at  iteration $\bar k$, condition \req{discrepance}   gives
\begin{eqnarray*}
\|y^\delta- F(x_{\bar k}^\delta)-J(x_{\bar k}^\delta)(x^\dagger-x_{\bar k}^\delta)\|& \le & 
\left( \frac{1+ c\|x^\dagger-x_{\bar k}^\delta\|}{\tau}+ c\|x^\dagger-x_{\bar k}^\delta\| \right) \|y^\delta-F(x_{\bar k}^\delta)\| ,
\end{eqnarray*}
and (\ref{locTR2}) yields   \req{gamma_libro} with
$
\theta_k=\displaystyle \frac{q\tau}{1+c(1+\tau)\|x^\dagger-x_{\bar k}^\delta\|}>1.
$
Then, Lemma \ref{sol_cond_q} yields $\lambda_{\bar k}^q>0$ and therefore $\lambda_{\bar k}\ge \lambda_{\bar k}^q$ is strictly positive.
Further,  by  a straightforward adaptation of   \cite[Proposition 4.1]{kns}, it follows
$$ 
\|x_{\bar k}^\delta-x^\dagger\|^2-\|x_{\bar k+1}^\delta-x^\dagger\|^2\ge  \frac{2(\theta_{\bar k}-1)}{\theta_{\bar k} \lambda_{\bar k}}\|F(x_{\bar k}^\delta)-y^\delta+J(x_{\bar k}^\delta) p_{\bar k}\|^2,
$$
and this implies 
$\|x_{\bar k+1}^\delta-x^\dagger\|<\|x_{\bar k}^\delta-x^\dagger\|$ and consequently $x_{\bar k+1}^\delta\in B_{\rho} (x_{\bar k}^\delta)$.  
Repeating the above  arguments,  by induction, we can prove that,   for $ \bar k<k<k_*$, condition \req{gamma_libro} holds,
$\lambda_k>0$, and 
\begin{equation}\label{monoton_1_new_noise}\qquad \ 
\|x_k^\delta-x^\dagger\|^2-\|x_{k+1}^\delta-x^\dagger\|^2\ge  \frac{2(\theta_k-1)}{\theta_k \lambda_k}\|F(x_{k}^\delta)-y^\delta+J(x_{ k}^\delta) p_k\|^2,
\end{equation}
with
$$\theta_k=\displaystyle \frac{q\tau}{1+c(1+\tau)\|x^\dagger-x_{ k}^\delta\|}.
$$
Thus $\|x_{k+1}^\delta-x^\dagger\|<\|x_k^\delta-x^\dagger\|$   
and     $\theta_{k+1}>\theta_k$ for $\bar k\le k<k_*$. 

{\it ({iii})} Proceeding as in the proof of point {\it (iii)} of Theorem \ref{lambdaTR_bound}, just replacing $x_k$ with $x_k^\delta$, we get that for  $ \bar k<k<k_*$, 
$\lambda_k<\bar \lambda$ with 
$$
\bar \lambda\le  \max \left \{  \frac{1}{ \gamma \omega },\, \frac{1}{C_{\min}}
\right \}.
$$
where $ \omega$ is obtained replacing $y$ with $y^{\delta}$ in (\ref{omega}).

{\it (iv)}  Summing up from $\bar k$ to $k_*-1$,  by (\ref{RQ}) and   (\ref{monoton_1_new_noise})  it follows 
$$
(k_*-\bar k)\tau^2 \delta^2 \le \sum_{k=\bar k}^{k_*-1} \|\Fdk\|^2\le\frac{\theta_{\bar k} \bar \lambda} {2(\theta_{\bar k}-1)q^2}\|x_{\bar k}^\delta-x^{\dagger}\|^2.
$$
Thus, $k_*$ is finite for $\delta>0$, and  convergence of $x^\delta_{k_*}$ to a solution of (\ref{prob}) as $\delta$ goes to 0 is shown in  \cite[Theorem 2.3]{hanke}.  
\end{proof}

\section{Numerical results} \label{exps}
In this section we  report on the performance of the regularizing 
trust-region   method
and make comparisons  with  the regularizing Levenberg-Marquardt method and 
a standard  version  of the trust-region  method.
The test problems are  ill-posed and with noisy data, and arise from the discretization of  
nonlinear Fredholm  integral equations of the first  kind
\begin{equation}\label{int_eq}
\int _0^1 k(t,s,x(s))ds=y(t), \quad t\in [0,1].
\end{equation}

 The integral equations considered  model  inverse problems from groundwater hydrology and geophysics.
 Their kernel is of the form
 \begin{equation}\label{kernel1}
 k(t,s,x(s))=\mbox{log}\left(\frac{(t-s)^2+H^2}{(t-s)^2+(H-x(s))^2}\right),
 \end{equation}
 see \cite[\S 3]{vogel}, or
 \begin{equation}\label{kernel2}
 k(t,s,x(s))=\frac{1}{\sqrt{1+(t-s)^2+x(s)^2}},
 \end{equation}
 see \cite[\S 6]{k}.
 The interval $[0,1]$ was discretized with  $n=64$ equidistant grid 
 points $t_i=(i-1) h$, $h=1/(n-1)$, $i=1, \ldots, n$. 
 Function $x(s)$ was approximated from the $n$-dimensional subspace of $H_0^1(0,1)$
 spanned by  standard piecewise linear functions. Specifically, we let 
 $s_j=(j-1) h$, $h=1/(n-1)$, $j=1, \ldots, n$, and looked for an approximation 
 $\hat x(s)=\sum_{j=1}^n \hat x_j \phi_j(s)$ where 
 $$
 \phi_1(s)=
 \left\{
 \begin{array}{ll}
 \frac{s_2-s}{h} & \mbox{if} \quad s_1\le s\le s_2 \\
 0   & \mbox{otherwise}
 \end{array}
 \right .,
 \qquad   
 \phi_n(s)=
 \left\{
 \begin{array}{ll}
 \frac{s-s_{n-1}}{h} & \mbox{if} \quad s_{n-1}\le s\le s_n \\
 0   & \mbox{otherwise}
 \end{array}
 \right .  ,
 $$
 and 
 $$
 \phi_j(s)=
 \left\{
 \begin{array}{ll}
 \frac{s-s_{j-1}}{h} & \mbox{if} \quad s_{j-1}\le s\le s_j,\\
 \frac{s_{j+1}-s}{h} & \mbox{if} \quad s_{j}\le s\le s_{j+1},\\
 0   & \mbox{otherwise}
 \end{array}
 \right . \qquad j=2,\ldots n-1.
 $$
Finally, the integrals $\int _0^1 k(t_i,s,\hat x(s))ds$, $1\le i\le n$, were  approximated
by the composite trapezoidal rule on the points $s_j$, $1\le j\le n$. 
The resulting discrete problems are square nonlinear systems \req{prob} with unknown 
$x=(\hat x_1, \ldots, \hat x_n)^T$.  We observe that $\hat x(s_j)=\hat x_j$; 
thus, the $j$-th component of $x$  approximates a solution of \req{int_eq} at $s_j$.
 
Two problems with kernel \req{kernel1} and two problems with kernel \req{kernel2} were considered
and built so that   solutions (later denoted as true solutions) are known. 
Concerning kernel \req{kernel1}, the first problem is given  in \cite[p. 46]{vogel}; it
admits as true continuous solutions the functions
$x_{true}(s)=c_1e^{d_1(s+p_1)^2}+c_2e^{d_2(s-p_2)^2}+c_3+c_4$ and 
$x_{true}(s)=2H-c_1e^{d_1(s+p_1)^2}-c_2e^{d_2(s-p_2)^2}-c_3-c_4$ 
where $H=0.2, c_1=-0.1, c_2=-0.075, d_1=-40, d_2=-60, p_1=0.4, p_2=0.67, c_3$ and $c_4$ 
are chosen such that $x_{true}(0)=x_{true}(1)=0$. The second problem was given in \cite[p. 835]{wy}
and it has true continuous solutions $x_{true}(s)=1.3s(1-s)+0.2$ and $x_{true}(s)=1.3s(s-1)$.

The third and fourth problems have   kernel  \req{kernel2};
the former  has  solutions $x_{true}(s)=1$ and $x_{true}(s)=-1,\;s\in [0,1]$,  
see \cite[p. 660]{k}, while the latter has the discontinuous  functions
\begin{equation}
x_{true}(s) = \begin{cases}   1 & \text{ if } 0 \leq s \leq \frac{1}{2} \\
0 & \text{ if } \frac{1}{2}< s \leq 1 \end{cases},\quad 
x_{true}(s) = \begin{cases}  - 1 & \text{ if } 0 \leq s \leq \frac{1}{2} \\
0 & \text{ if } \frac{1}{2}< s \leq 1 \end{cases}
\end{equation}
as the true solutions, \cite[p. 662]{k}.

The nonlinear systems arising  from the  discretizations of the four  problems 
are denoted as {\tt P1},  {\tt P2}, {\tt P3} and {\tt P4} respectively, while  $x^\dagger\in \mathbb{R}^n$
denotes    a solution of the discretized problems.
Given the error level $\delta$, the exact data $y$ was perturbed by 
normally distributed values with mean $0$ and variance $\delta$  using
the {\sc Matlab} function {\tt randn}.

All procedures were implemented in {\sc Matlab} and run using {\sc Matlab  2014}b on an Intel Core(TM) i7-4510U
2.6 GHz, 8 GB RAM; the  machine precision is $\epsilon_m\approx 2\cdot 10^{-16}$. 
The Jacobian of the  nonlinear function $F$   was computed by finite differences. 
The parameter $q$ used in \req{seculare_q} and in \req{RQ} 
was set equal to $1.1/\tau$ and the discrepancy principle \req{discrepance}  
with $\tau=1.5$ was used as the stopping criterion. A maximum number of 300 iterations was allowed and a failure  
was declared when this limit was exceeded.

In the implementation of the  regularizing  trust-region method, 
Step 3 in Algorithm \ref{algoTR} was performed setting $\eta=\frac{1}{4}$, $\gamma= \frac{1}{6}$.
Then,  in Step 2 the trust-region radius  was updated 
as follows 
\begin{eqnarray}
& \Delta_0  =\mu_0\|F_0\|,   &  \   \mu_0=10^{-1},\label{tr_algo1}\\
&  \Delta_{k+1} =\mu_{k+1}\vert\vert F(x_{k+1}^\delta)\vert\vert ,  \qquad & \ \mu_{k+1} = \begin{cases} \displaystyle \frac{1}{6}\mu_k & \text{ if } q_k<q, \label{tr_algo2}\\
  2\mu_k & \text{ if } q_k>\nu q, \end{cases},
\end{eqnarray}
with $q_k=\displaystyle \frac{\vert\vert J(\xkd)p_k+F(\xkd)\vert\vert}{\vert\vert F(\xkd)\vert\vert }$, and $\nu=1.1$.
The maximum and minimum values for $\Delta_k$ were set to
$\Delta_{\max}=10^4$ and  $\Delta_{\min}=10^{-12}$.
This updating strategy turned out to be efficient 
in practice and  was based on the following considerations.
Clearly, $\Delta_k$ is cheaper to compute than the upper bound in (\ref{radius})  and 
preserves the property of converging to 
zero as $\delta$ tends to zero and a solution of problem (\ref{prob_noise}) is approached. Further,
$\Delta_k$ is adjusted taking into account the  $q$-condition and  by monitoring  the value $q_k$;
therefore, if the $q$-condition was not satisfied at the last computed iterate $\xkd$, 
it is reasonable to take a smaller
radius than in the case where the $q$-condition was fulfilled.

The computation of the parameter $\lambda_k$ was performed  applying Newton method to the equation
\begin{equation}\label{sec_TR}
\psi(\lambda)=\frac{1}{\|p(\lambda)\|}-\frac{1}{\Delta_k}=0,
\end{equation}
and each  Newton iteration requires the Cholesky factorization of a shifted  
matrix of the form $B_k+\lambda I$  \cite{cgt}.
Typically high accuracy in the solution of 
the above scalar equations is not needed \cite{bcgmt, cgt} and this 
fact was  experimentally verified   also for our test problems. 
Hence, after extensive numerical 
experience, we decided to terminated the Newton process   
as soon as  $|\Delta_k-\|p(\lambda)\||\le 10^{-2} \Delta_k$.

In our implementation of the standard trust-region method, we chose 
the trust-region radius  accordingly to technicalities well-known in literatures, see 
e.g. \cite[\S 6.1]{cgt} and \cite[\S 11.2]{nw}.   In particular, we set $\Delta_0=1$,
$$
\Delta_{k+1}=\left \{ 
\begin{array}{lll}
&\displaystyle \frac{\|p_k\|}{4}, & \quad \mbox{ if } \rho_k<\frac{1}{4},\\
&\Delta_k,  & \quad \mbox{ if }  \frac{1}{4}\le \rho_k\le\frac{3}{4},\\
& \min\{2\Delta_k,\Delta_{\max}\},& \quad \mbox{otherwise},
\end{array}
\right.
$$
with $\Delta_{\max}=10^4$ and chose $\Delta_{\min}=10^{-12}$
as the   minimum values for $\Delta_k$.

Finally the Levenberg-Marquardt approach was implemented 
imposing condition (\ref{seculare_q}) and solving 
\req{sec_new}  by Newton method. In order to find an accurate solution for (\ref{seculare_q})
it was necessary to use  a tighter  tolerance, equal to $10^{-5}$, than that used in the trust-region algorithm.

Our  experiments were made   varying the noise level $\delta$ on the data $\yd$.
Tables \ref{table1} and  \ref{table2} display  the results obtained by the regularizing
trust-region algorithm with noise $\delta=10^{-4}$ and  $\delta= 10^{-2}$ respectively.
Runs for four different initial guesses $ x_0 $ are reported in the tables. 
For problems P1 and P2 the initial guesses are $x_0=0e,-0.5e,-e,-2e$ and $x_0=0e,0.5e,e,2e$ respectively, where
$e$ denotes the vector $e=(1,\ldots,1)^T$.  
For problem P3 the initial guess was chosen as the vector $x_0(\alpha)$ with $j$-th component given 
by $(x_0(\alpha))_j=g_{\alpha}(s_j)$  for $j=1,\dots,n$, where $g_\alpha(s)=(-4\alpha+4)s^2+(4\alpha-4)s+1$, 
and $s_j$ being the grid points in $[0,1]$. We have used the following values of $\alpha$, 
$\alpha=1.25,1.5,1.75,2$. 
For problem P4 the initial guess $x_0(\beta, \chi)$   has components 
$(x_0(\beta,\chi))_j=g_{\beta,\chi}(s_j)$ for $j=1,\dots,n$  with $g_{\beta,\chi}=\beta-\chi s$ and $(\beta,\chi)=(1,1),(0.5,0),(1.5,1),(1.5,0)$. 
In the tables we report: the initial guesses (for increasing distance from the true solutions) the number of iterations 
{\tt it} performed; the final norm of function $F$;
the number of function evaluations {\tt nf} performed; the rounded
average number {\tt cf} of Cholesky factorizations per iteration. 
To assess the quality of the results obtained, we measured  the distance 
between the final iterate $x_{k^*}^\delta$ and the  true solution approached; in particular 
${\tt e_I}=\max_{2\le j\le n-1} |x_{true}(s_j)-(x_{k^*}^\delta)_j|$ is the maximum absolute value  of the difference between the 
components associated to internal points $s_j\in(0,1)$, while 
${\tt e_T}=\max_{1\le j\le n} |x_{true}(s_j)-(x_{k^*}^\delta)_j|$  is the maximum absolute value  of the difference between the 
components associated to points $s_j$ including  the end-points of the interval $[0,1]$.
The symbol $``*''$ indicates that either the procedure failed to satisfy the discrepancy principle within the 
prescribed maximum number of iteration, or the final $x_{k^*}^\delta$ was not an approximation of one of the
true solutions described above.
 \begin{center}		
 	\begin{table}
 		{\small
 			\begin{tabular}{lr|cccccc|ll}
 				\hline\hline
 				\mbox{ Problem} &	&  \multicolumn{6}{c|}{RTR}  &    \multicolumn{2}{c}{RLM}  \\ 
 				&  $x_0$  & {\tt it}  & ${\tt \vert\vert F \vert\vert}$  & {\tt nf} & {\tt cf} & ${\tt e_I}$  & ${\tt e_T}$ & ${\tt e_I}$& ${\tt e_T}$\\ \hline\hline
 				{\tt P1} & $0\,e$    &  \z43   &      1.3e$-$4  &   \z44     &  5  & 5.5e$-3$ &  5.5e$-3$ & 4.5e$-3$ & 4.5e$-3$ \\
 				&		$-0.5\, e$ &   \z 63  &  	  1.2e$-$4 		&  \z71  &   5  & 3.2e$-2$ & 7.9e$-2$& 3.0e$-2$ & 7.1e$-2$  \\  
 				&		$-1\, e$    &  \z 82   &  1.4e$-$4  		&  \z94    &    4 & 3.4e$-2$  & 8.4e$-2$ & 4.0e$-2$ & 7.2e$-2$\\  
 				&		$-2\, e$   &  115   &  1.5e$-$4   		&   138   &   4 & 3.4e$-2$ & 8.6e$-2$  & 2.9e$-2$ & 6.1e$-2$\\
 			 \hline 
 				
 				{\tt P2} &$0\, e$ &   \z54  &   1.2e$-$4     &  \z55  &  5  & 7.4e$-3$ &  7.4e$-3$ & * & * \\
 				& $0.5\,e$    &  \z56   &      1.4e$-$4 &    \z59   & 5   & 1.1e$-2$ & 1.3e$-2$ & * & *   \\
 				&$1\, e$ & \z73  &  	 1.4e$-$4 		&  \z84    &  4 &   1.0e$-2$ & 1.3e$-2$  & 7.3e$-3$ & 8.3e$-3$  \\  
 				&$2\, e$   &  118   &  1.4e$-$4 	&   138   &  4  &  9.3e$-3$ & 1.1e$-2$  & 4.8e$-3$ & 4.8e$-3$  \\  
 			\hline
 				
\hline
 				
 					{\tt P3}&$x_0(1.25)$     & \z35   &  1.4e$-$4   & \z36    & 3  &  1.2e$-2$ &  1.2e$-2$ & 3.1e$-3$  & 3.1e$-3$  \\
 					&$x_0(1.5) $  &  \z43  &  	 1.4e$-$4 	&  \z44  &   3   &  5.1e$-2$ & 5.1e$-2$ & 6.2e$-2$  & 6.2e$-$2\\  
 					&$x_0(1.75) $    &   \z45   &  1.3e$-$4   &    \z46  &  3  &  3.2e$-1$ & 3.2e$-1$  & 3.1e$-1$ & 3.1e$-$1  \\  
 					& $x_0(2)$   &  \z65   & 1.4e$-$4 &   \z 71   &  3   &  4.6e$-1$ & 4.6e$-1$  & 3.8e$-1$ & 3.8e$-1$ \\
 				
 				\hline
 				
 					{\tt P4}&$x_0(1,1)$   &\z68  &  1.5e$-$4   &\z82    & 3   &  4.8e$-1$ & 4.8e$-1$  & * & *   \\
 					&$x_0(0.5,0)$  & \z64   &  1.5e$-$4 	&  \z75  &  3 &  4.9e$-1$  &  4.9e$-1$  & 4.7e$-1$  & 4.7e$-1$ \\
 					&$x_0(1.5,1)$  &  \z69   & 1.5e$-$4 	& \z78  &   3  &  5.1e$-1$ & 5.1e$-1$  & 4.8e$-1$ & 4.8e$-1$  \\  
 					
 					&$x_0(1.5,0)$   &  \z68  &   1.5e$-$4 &  \z78  & 4  &  5.2e$-1$ & 7.1e$-1$  & 5.1e$-1$ & 6.3e$-1$ \\ \hline\hline

 			\end{tabular}
 		}
 		\caption{Results obtained by the regularizing trust-region method and the regularizing
 			Levenberg-Marquardt method with noise  $\delta=10^{-4}$ and varying initial guesses.}\label{table1}
 	\end{table}
 \end{center}
 \begin{center}		
 	\begin{table}
 		{\small
 			\begin{tabular}{lr|cccccc|ll}
 				\hline\hline
 				\mbox{Problem} &	&  \multicolumn{6}{c|}{RTR}  &    \multicolumn{2}{c}{RLM}  \\ 
 				&   $x_0$  & {\tt it}  & {\tt $\vert\vert F \vert\vert$}  & nf & {\tt cf} & {\tt $e_I$}  & {\tt $e_T$} & {\tt $e_I$}& {\tt $e_T$}\\ \hline\hline
 				{\tt P1}    & 0$\,e$  & \z20  &   1.5e$-$2  &  {\z21}    & {6} &  1.9e$-2$ &  1.9e$-2$ & 1.8e$-2$ & 1.8e$-2$   \\
 				&$-0.5\, e$ & \z 29  & { 1.0e$-$2}& {\z 30}  &  {6} &  2.2e$-2$ & 3.1e$-1$  & 2.1e$-2$ & 3.1e$-1$\\  
 				&$-1\, e$   & {\z35}   & { 1.4e$-$2} &  {\z 36}  &  {5}  &  3.6e$-2$ &  6.1e$-1$  & 3.3e$-2$ & 6.1e$-1$ \\  
 				&$-2\, e$   & {\z 40}   & { 1.3e$-$2}  		&  {\z  41}   &  {5}  &  4.9e$-2$ & 1.2e$+$0  & 4.5e$-2$ & 1.2e$+$0 \\
 				 \hline

 				{\tt P2} & $0\,e$   & {\z 30}  &     { 1.4e$-$2} &  {\z  31}   &{5} &  6.9e$-3$&  1.3e$-2$ & * & *     \\
 				&$0.5\, e$ & {\z 25} & {  1.4e$-$2}		& {\z 26}  &  {5} &  1.7e$-2$&  2.1e$-1$ & * & *  \\  
 				&$1\, e$  & {\z 29}   & {  1.4e$-$2} &  {\z 30}  &   {5} &  3.8e$-2$&  5.4e$-1$  & 1.3e$-1$ & 5.2e$-1$   \\  
 				&$2\, e$  & {\z 37}   & { 1.4e$-$2}  &  {\z  39}   &  {5} &  5.5e$-2$ &  1.2e$+$0  & 2.2e$-1$ & 1.1e$+$0   \\
 				\hline
 				
%
 					{\tt P3}&$x_0(1.25)$  &{\z 15}  &  {1.2e$-$2}   &{\z 16}   & {4} & 1.5e$-1$ &  1.5e$-1$ & 1.5e$-1$  & 1.5e$-1$    \\
 					&$x_0(1.5)$ & {\z 17} &  	{  1.4e$-$2}		& {\z  18}  &  {4} &  3.2e$-1$ &  3.2e$-1$  & 3.2e$-1$ & 3.2e$-1$\\  
 					&$x_0(1.75)$ & {\z 19}   & { 1.4e$-$2} 	&  {\z  20}  &   {4} &  5.0e$-1$&  5.0e$-1$  & 5.1e$-1$ & 5.1e$-1$  \\  
 					&$x_0(2)$ & {\z 22}   & {  1.5e$-$2}  		&  {\z  23}   &  {4} &  6.9e$-1$ &  6.9e$-1$ & 7.0e$-1$ & 7.0e$-1$ \\
 				\hline
 				
	{\tt P4}&$x_0(1,1)$&{\z 17}  &  {  1.4e$-$2}   &{\z 18}   & {5} &  5.7e$-1$ &  5.7e$-1$ & 5.4e$-1$    & 5.4e$-1$    \\
&$x_0(0.5,0)$  & {\z20}   & { 1.3e$-$2}  		&  {\z 21}   &  {4}&  5.5e$-1$ &  5.5e$-1$  & * & *    \\
	&$x_0(1.5,1)$  & {\z 22}   & { 1.4e$-$2} 	&  {\z 23}  &   {4}  &  5.1e$-1$ &  5.1e$-$1  & 5.0e$-1$ & 5.0e$-1$  \\  
	
	& $x_0(1.5,0)$&  {\z 26}  &   { 1.5e$-$2}    & {\z 27} & {4} & 5.2e$-1$  & 8.8e$-1$   & * & *  \\ \hline\hline 	
		\end{tabular}
 		}
 		\caption{Results obtained by the regularizing trust-region method and the regularizing
 			Levenberg-Marquardt method with noise $\delta=10^{-2}$ and varying initial guesses.}\label{table2}
 	\end{table}
 \end{center}
 
Tables \ref{table1} and \ref{table2} show that the regularizing trust-region method solves all
the tests. By Step 3 of our Algorithm \ref{algoTR}, the difference between the
number of function evaluations and the number of trust-region iterations,
if greater than one, indicates the number of trial iterates that were rejected
because a sufficient reduction on $\Phi$ was not achieved. 
We observe that in 20 out of 32 runs, all the iterates generated were accepted;
this occurrence seems to indicate that the trust-region updating rule works well in practice. 

Further insight on the trust-region updating rule (\ref{tr_algo1})-(\ref{tr_algo2})
can be gained analyzing the regularizing properties 
of the implemented trust-region strategy. First, we verified numerically that, 
though not explicitly enforced, the $q$-condition is satisfied in most of the iterations.
As an illustrative example, we consider problem P2 with 
$\delta=10^{-4}$ and $x_0=0e$  and,  in the left plot in Figure \ref{fig:figure9},
we display the values 
$q_k=\frac{\vert\vert J(\xkd)p_k+F(\xkd)\vert\vert}{\vert\vert F(\xkd)\vert\vert }$ 
at the trust-region iterations, marked by an asterisk,  and the value 
$q=1.1/\tau\approx0.733$  fixed in our experiments, depicted by a solid line.  
We observe that, even if we have not imposed the $q$-condition, it is satisfied 
at  most of the iterations. 
The plot on the right of  Figure \ref{fig:figure9}  shows
a monotone decay of the error between $\xkd$ and $x^\dagger$ through the iterations,
which results to  be in accordance with the theoretical results in Theorem \ref{lambdaTR_bound_noise}.
The regularizing properties of the implemented trust-region scheme are also shown in 
Figure \ref{fig:figure10} where, for each test problem we plot the error    
 $\vert\vert x_{k^*}^\delta-x^{\dagger}\vert\vert$ for decreasing noise levels;
 it is evident that, in accordance with theory,
the error decays as the noise level decreases.

\begin{figure}[h]
\centering
\includegraphics[width=2.4in,height=2in]{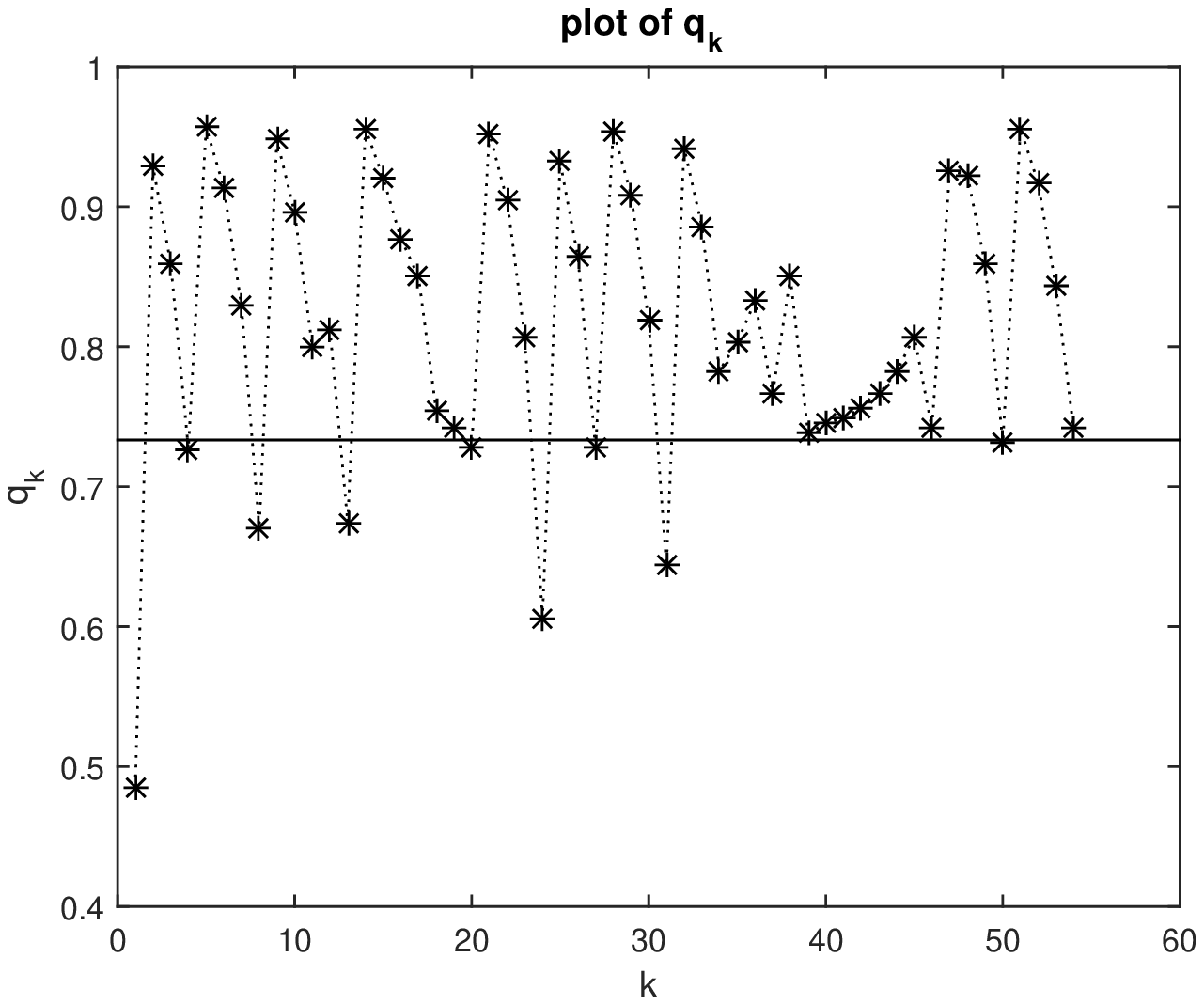}
\includegraphics[width=2.4in,height=2in]{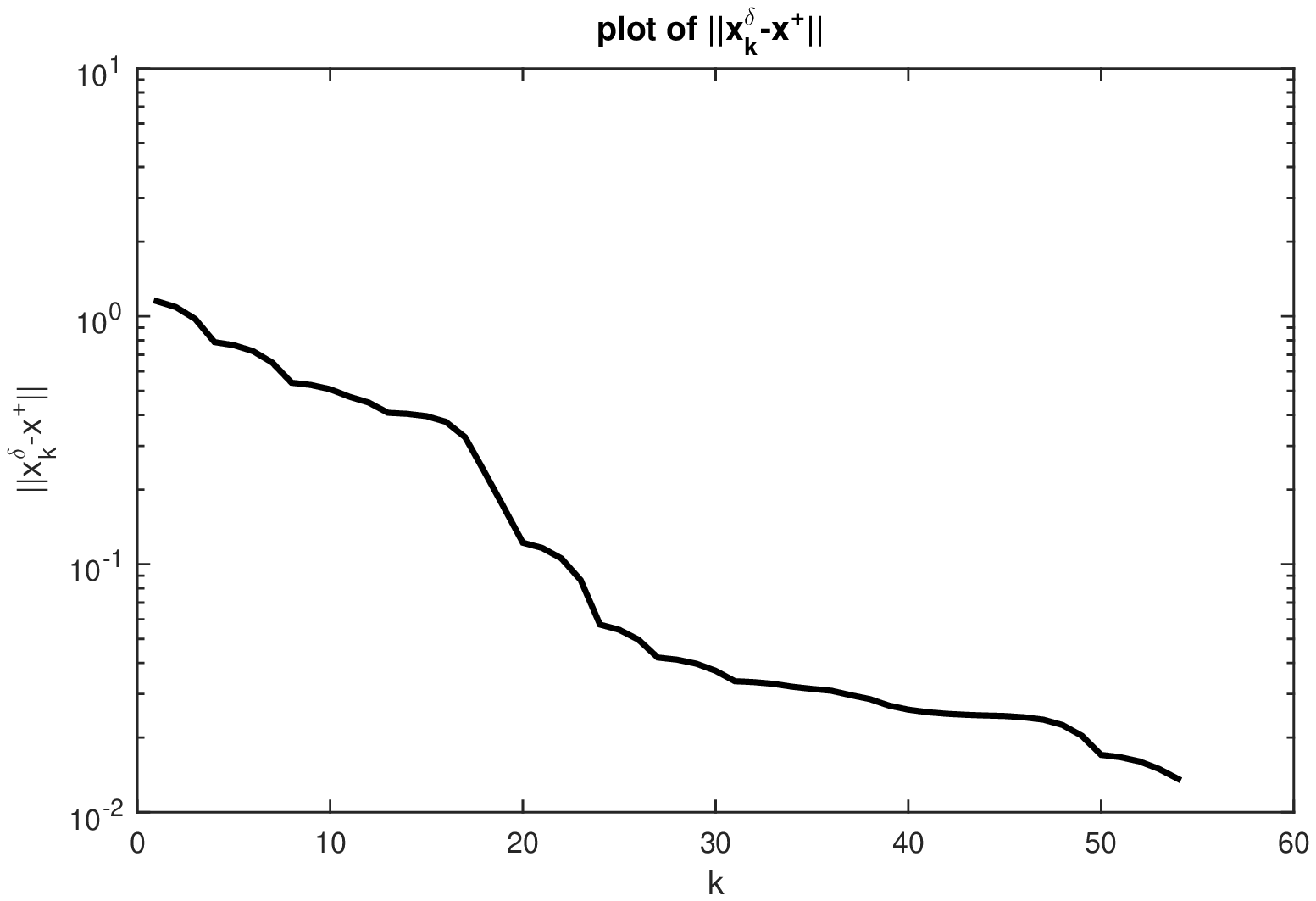}
\caption{Regularizing trust-region applied to P2, $x_0=0e$, $\delta=10^{-4}$: values $q_k=\frac{\vert\vert J(\xkd)p_k+F(\xkd)\vert\vert}{\vert\vert F(\xkd)\vert\vert }$ (marked by an asterisk) and value of $q=1.1/\tau$ (solid line) versus the  iterations (on the left);  semilog plot  of the error  $\vert\vert x_k^{\delta}-x^{\dagger}\vert\vert$  versus the  iterations (on the right).}
\label{fig:figure9}
\end{figure}

\begin{figure}[h]
\centering
\includegraphics[width=2.4in,height=2in]{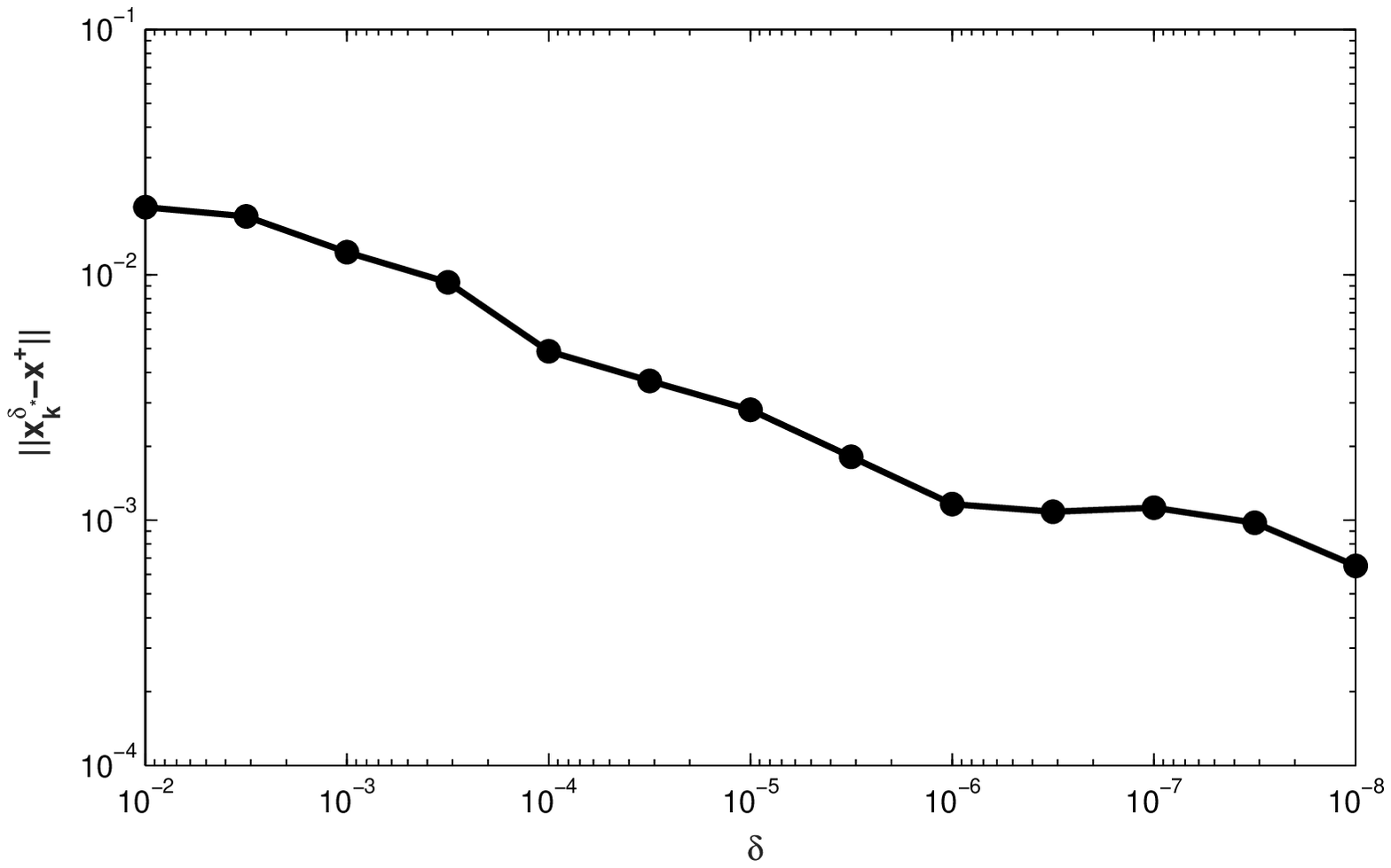}
\includegraphics[width=2.4in,height=2in]{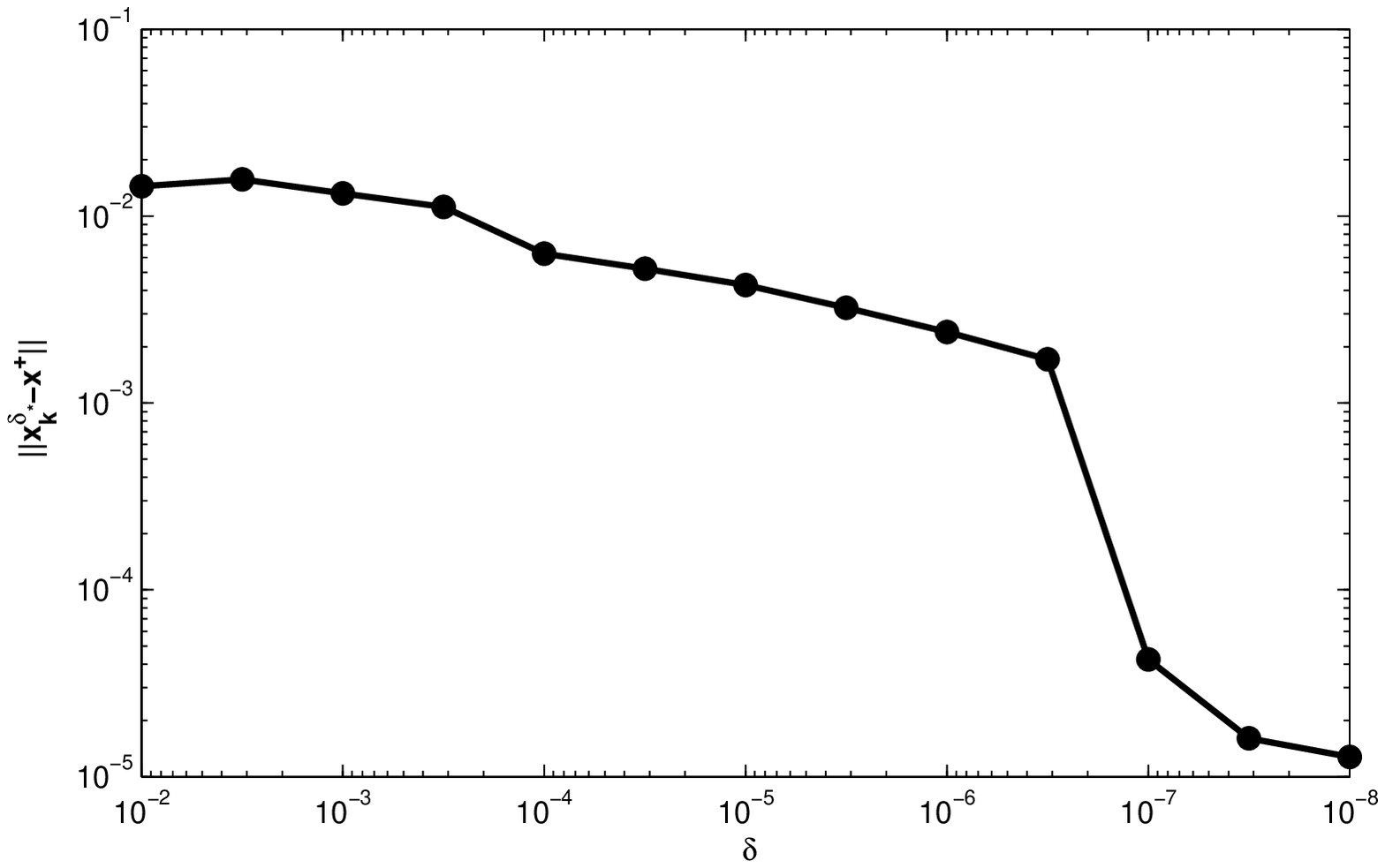}
\includegraphics[width=2.4in,height=2in]{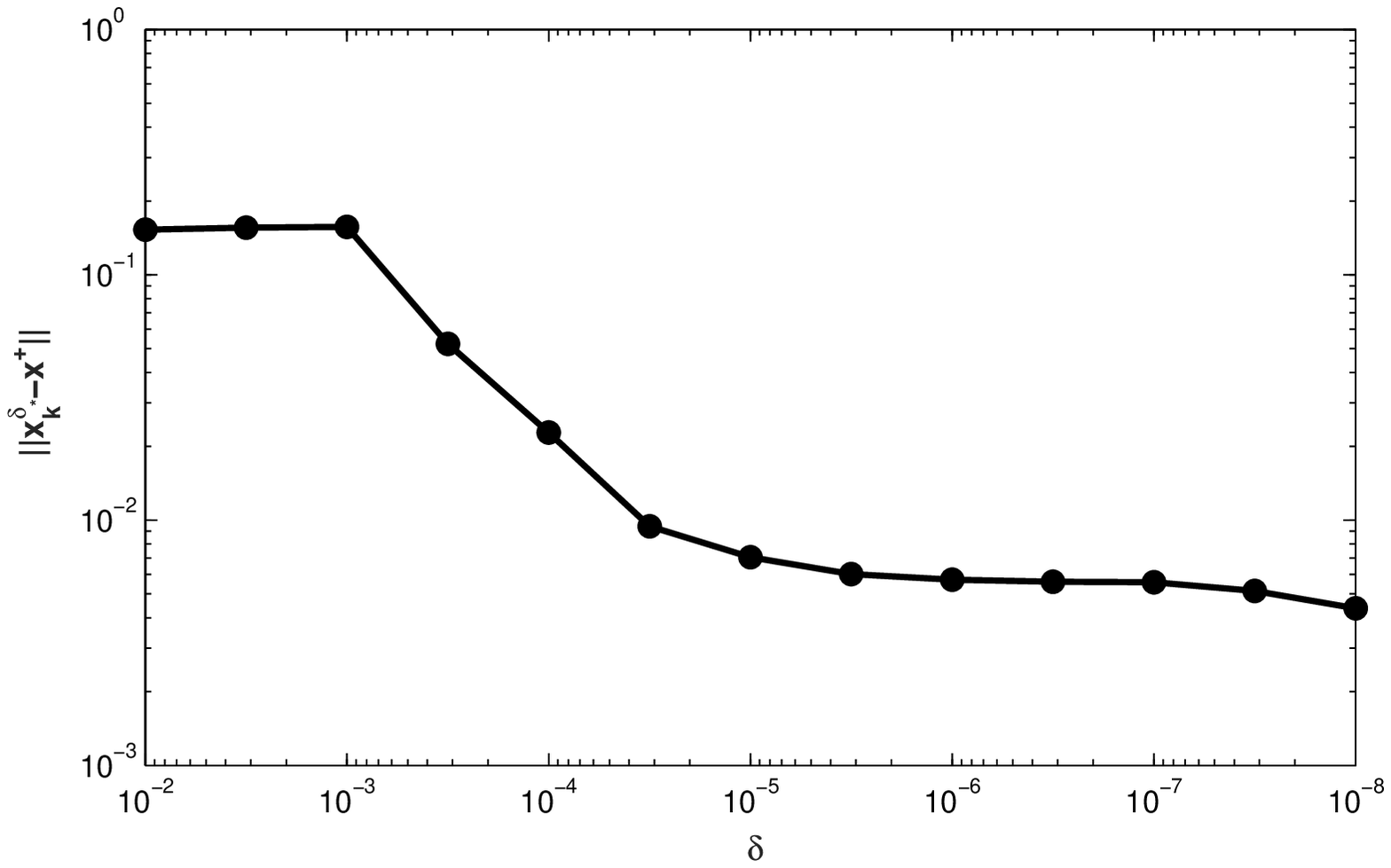}
\includegraphics[width=2.4in,height=2in]{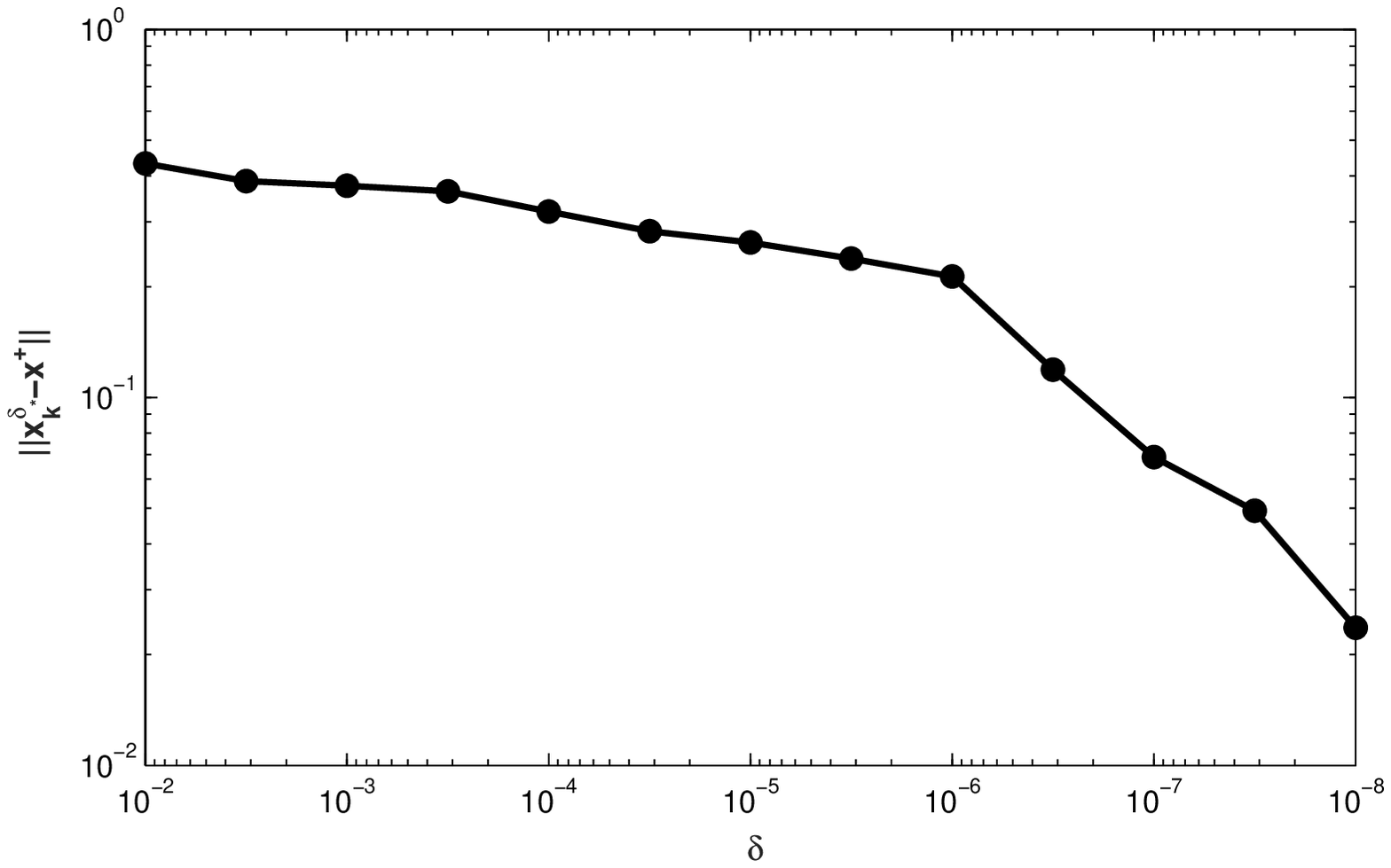}
\caption{Regularizing trust-region applied to P1, $x_0=0e$ (top left), P2, $x_0=0e$ (top right), P3, 
$x_0=x_0(\alpha)=x_0(1.25)$ (lower left) and to P4, $x_0=x_0(\beta, \chi)=x_0(0.5,0)$ (lower right):
log plot of the error  $\vert\vert x_{k^*}^{\delta}-x^{\dagger}\vert\vert$  
versus the noise $\delta$.}
\label{fig:figure10}
\end{figure}

Let now compare the regularizing trust-region and  Levenberg-Marquardt procedures. 
On runs successful for both methods, the two methods provide solutions
of similar accuracy and such an accuracy increases with the vicinity 
of the  initial guess to the true solution; as an example Figure \ref{fig:figure20} shows the solutions computed by 
the two methods for   problems P1 and  P3  for $\delta=10^{-2}$. On the other hand, 
for  large noise $\delta$ and initial guesses farther from the true solution,  for both methods the accuracy at the endpoints of the interval
$[0,1]$ may deteriorate; for this occurrence we refer to Table \ref{table2} and runs on 
problems P1 and P2.
Concerning failures,  in 7 runs out of 32 the Levenberg-Marquardt algorithm does not act as a regularizing method as the generated sequence  
approaches a solution of the noisy problem.
In Figure \ref{fig:figure4} we illustrate two unsuccessful runs of the Levenberg-Marquardt method;
approximated solution computed by the regularizing trust-region and  Levenberg-Marquardt procedures are shown for 
runs on problems P2 and P4.

\begin{figure}[h]
	\centering
	\includegraphics[width=2.4in,height=2in]{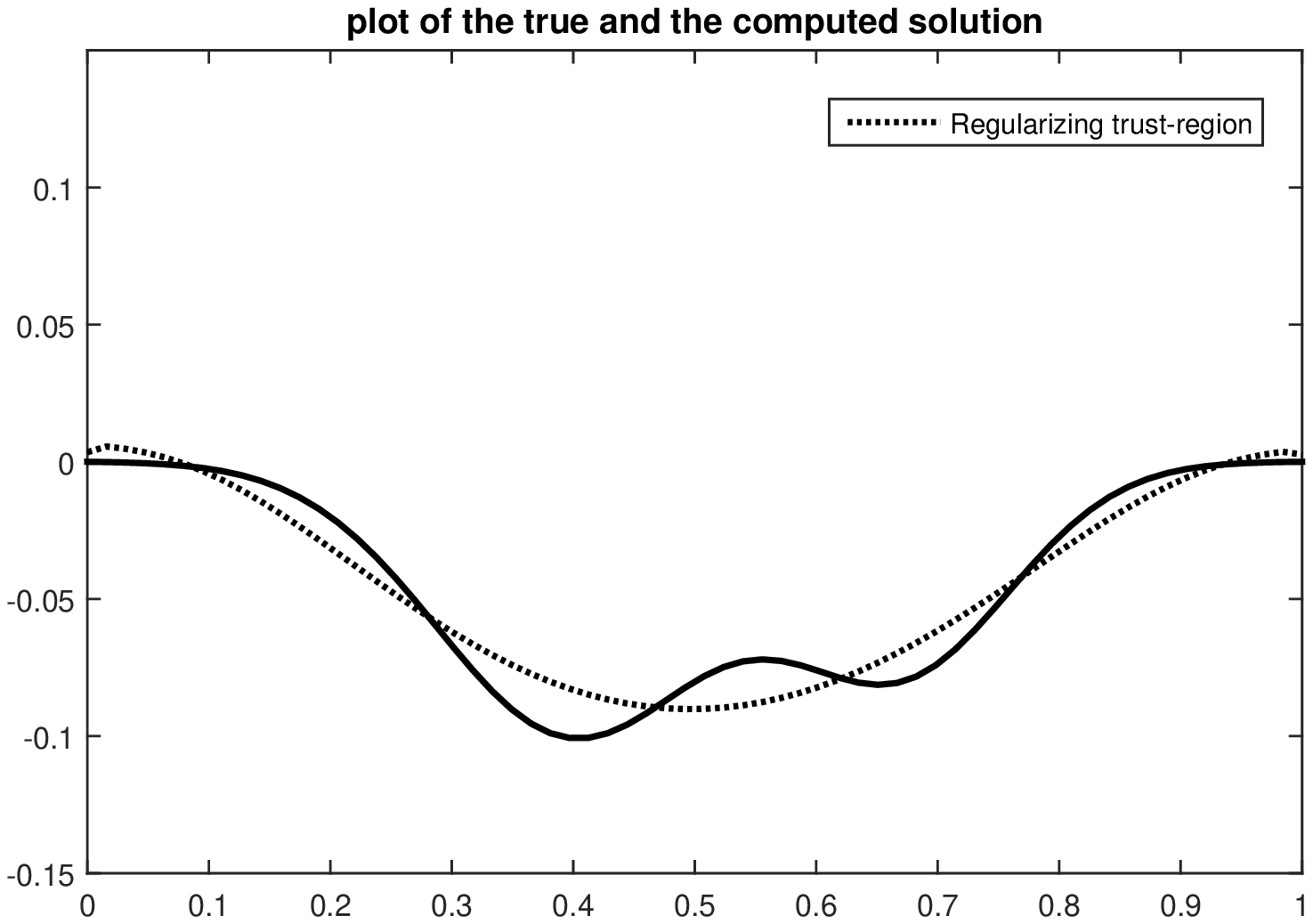}
	\includegraphics[width=2.4in,height=2in]{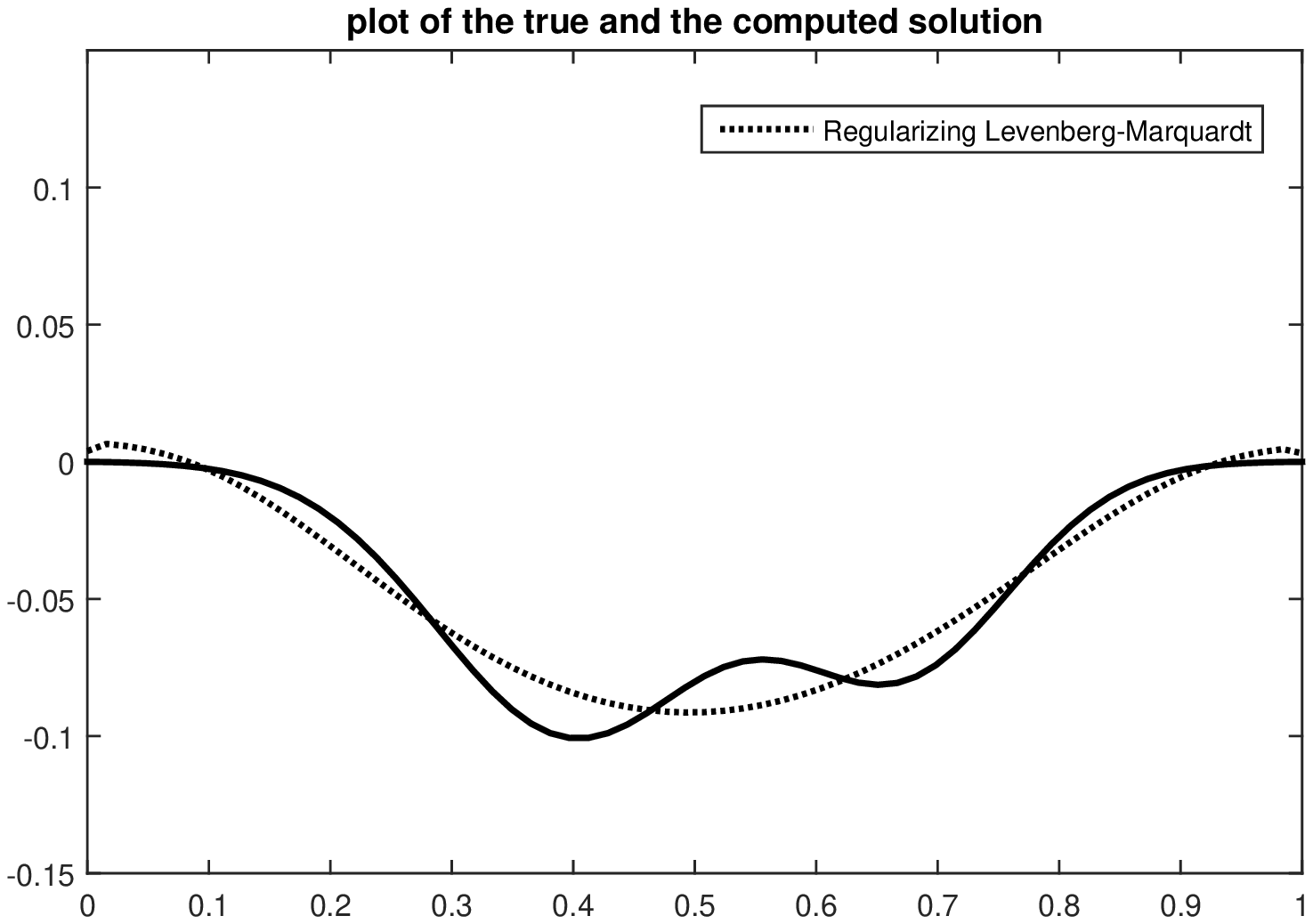}
	\includegraphics[width=2.4in,height=2in]{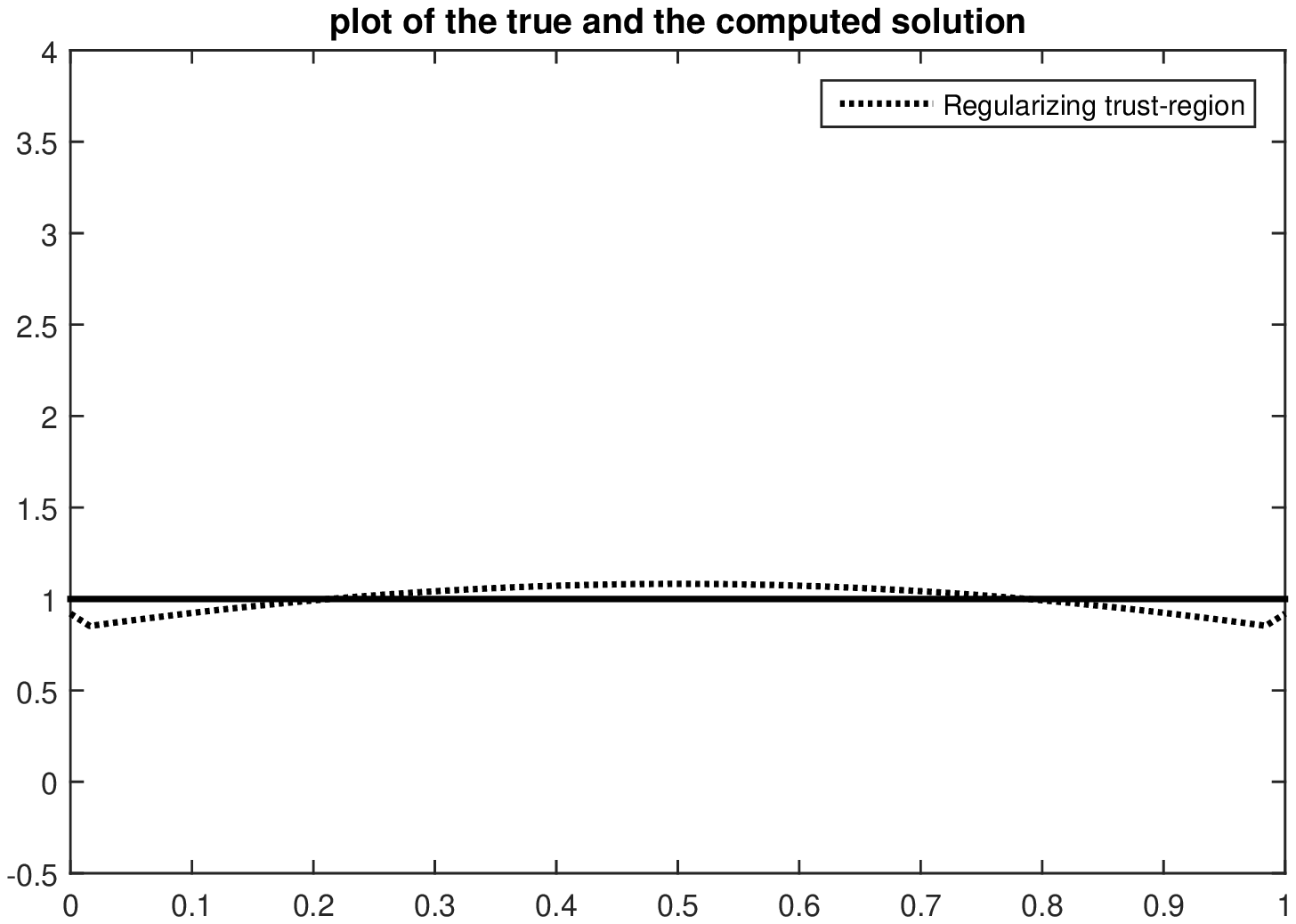}
	\includegraphics[width=2.4in,height=2in]{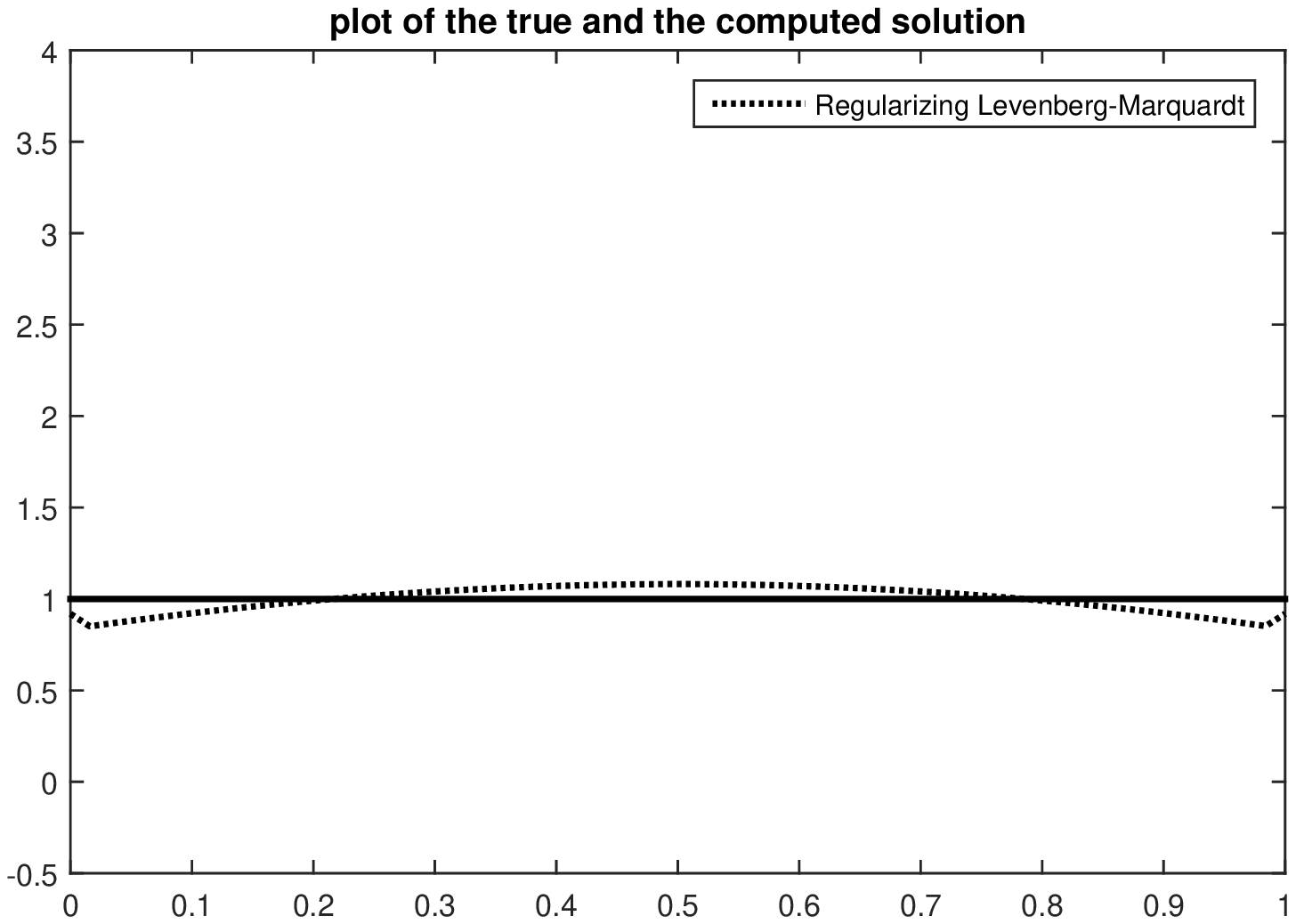}
	\caption{Regularizing trust-region (left) and regularizing Levenberg-Marquardt (right), true solution (solid line) and approximate solutions (dotted line). Upper part: P1, $\delta=10^{-2}$, $x_0=0e$; lower part: P3, $\delta=10^{-2}$, $x_0=x_0(\alpha)=x_0(1.25)$.}
	\label{fig:figure20}
\end{figure}

\begin{figure}[h]
	\centering
	\includegraphics[width=2.4in,height=2in]{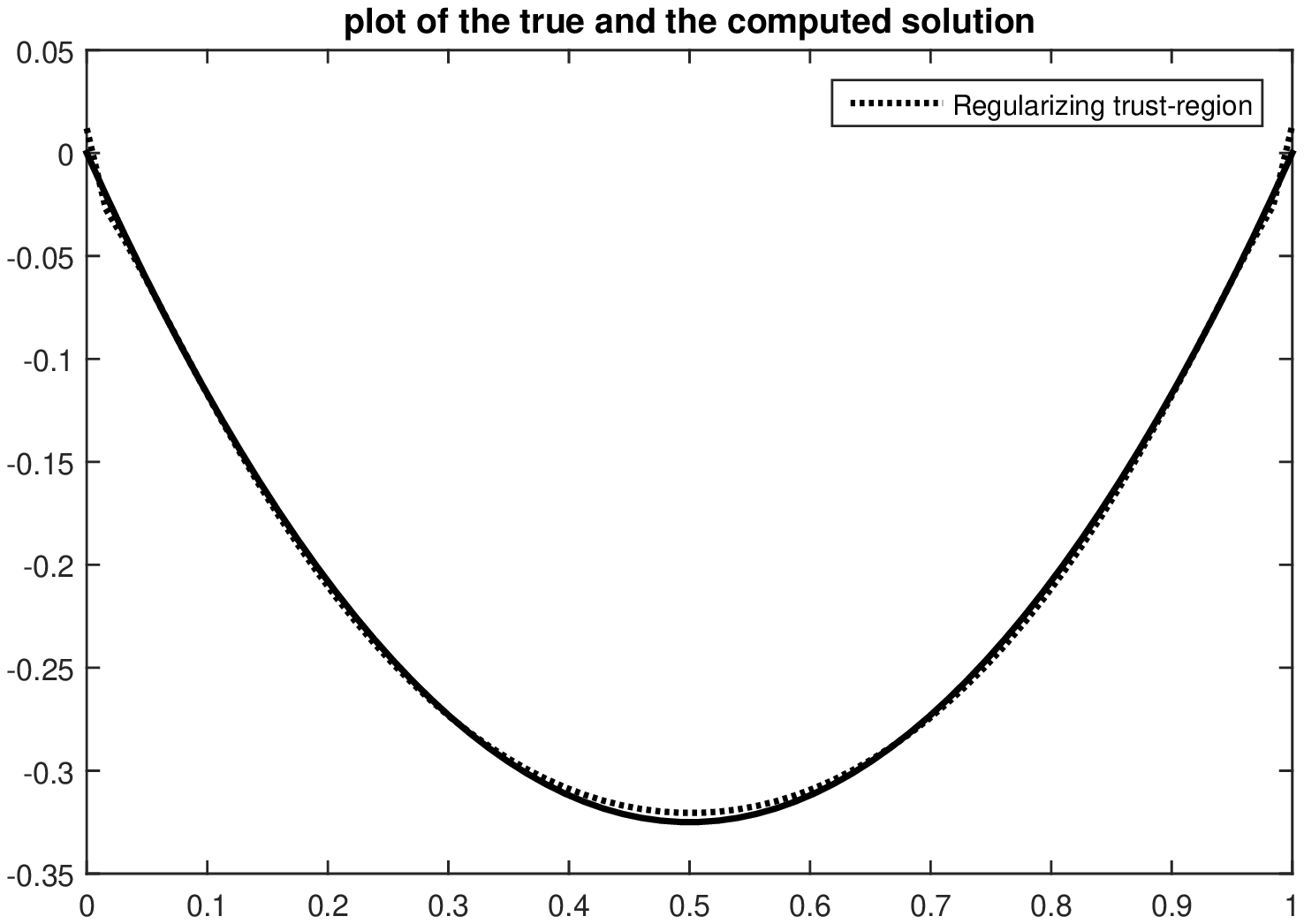}
	\includegraphics[width=2.4in,height=2in]{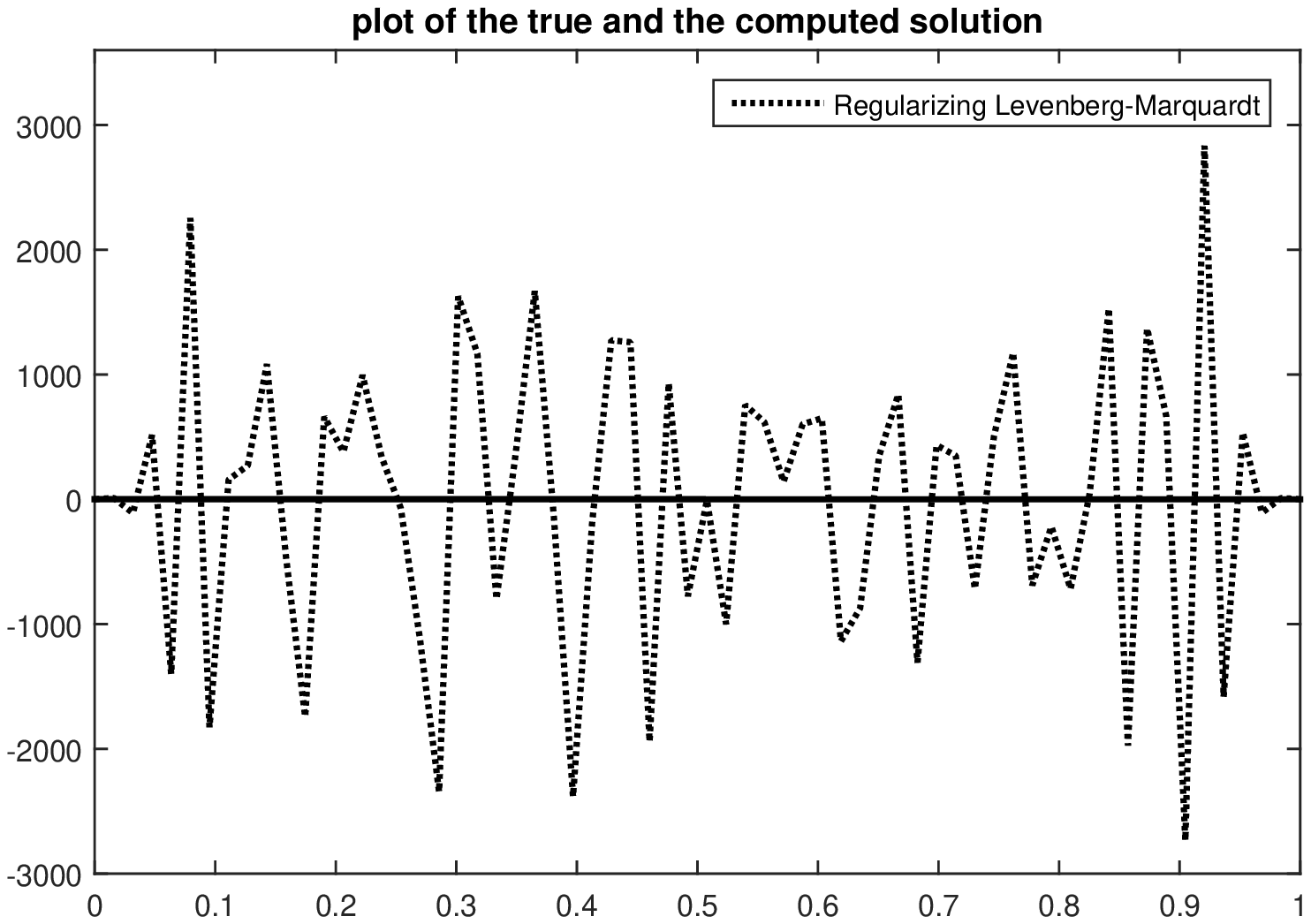}
	\includegraphics[width=2.4in,height=2in]{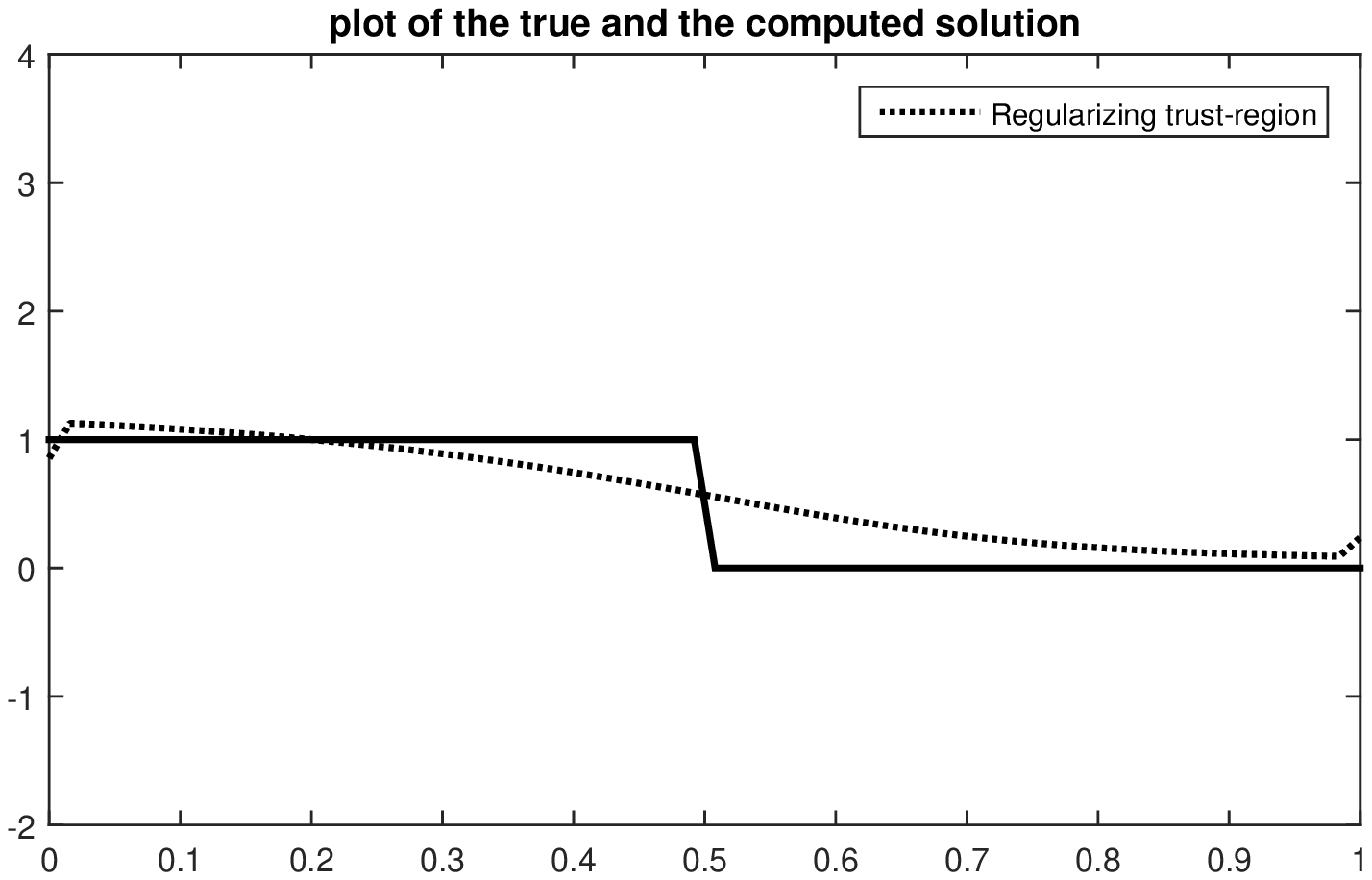}
	\includegraphics[width=2.4in,height=2in]{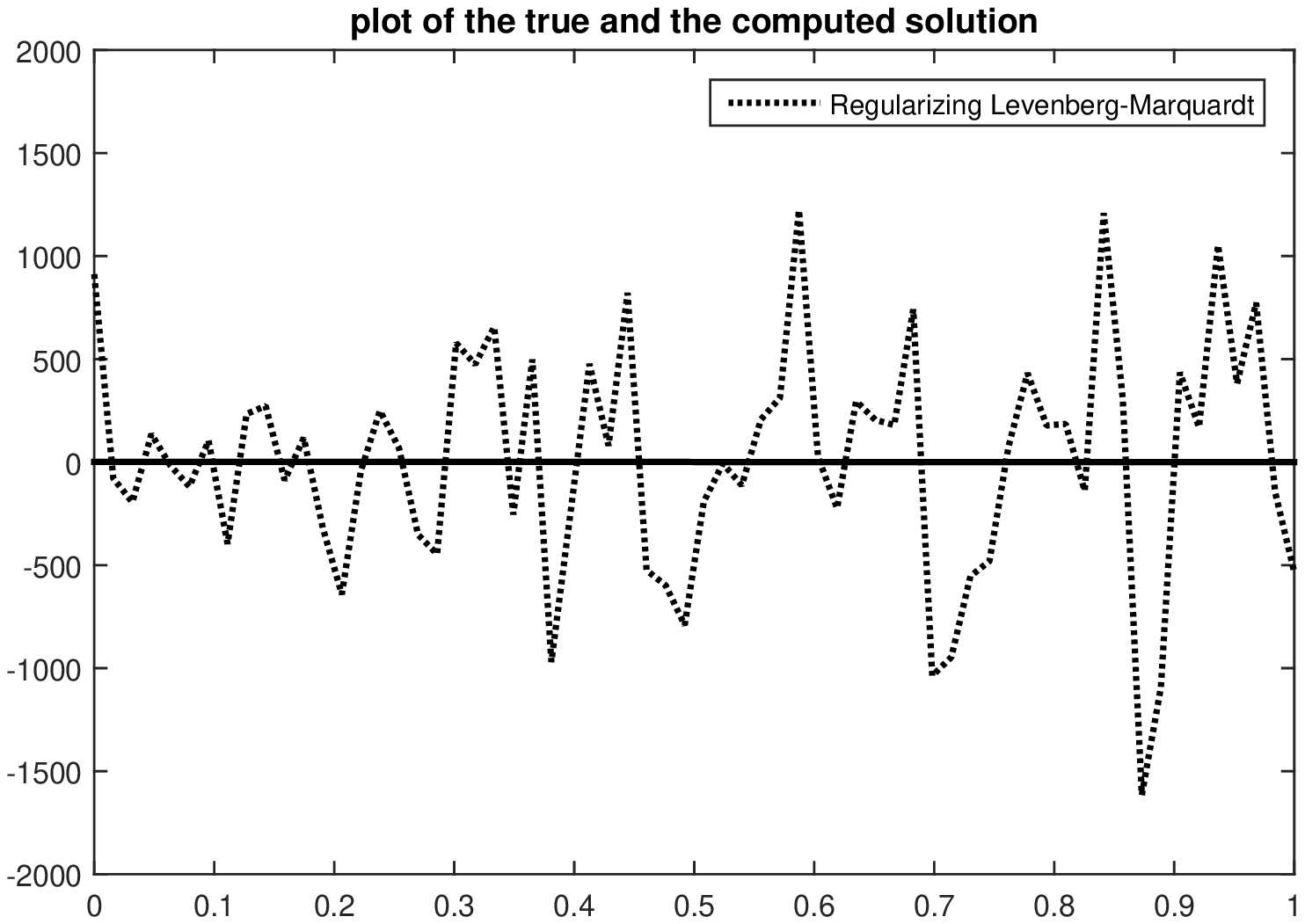}
	\caption{True solution (solid line) and approximate solutions (dotted line) computed by the regularizing  trust-region method (on the left) and the regularizing Levenberg-Marquardt method (on the right). Upper part: problem P2, $\delta=10^{-2}, x_0=0e$; 
	lower part:    problem P4,  $\delta=10^{-2}, x_0=x_0(\beta, \chi)=x_0(0.5,0)$.
	}
	\label{fig:figure4}
\end{figure}

The overall experience on  the Levenberg-Marquardt algorithm seems to indicate that 
the  use of the $q$-condition is more flexible  than condition (\ref{seculare_q})
and provides stronger regularizing properties. In order to
support this claim, in Figure \ref{fig:LM_q} we  report 
four solution approximations computed by the Levenberg-Marquardt algorithm
for varying values of $q$, i.e.  $q=0.67,0.70,0.73,0.87$.  It is evident that
the method is highly sensitive to the choice of the parameter $q$ and the quality of the solution approximation does not steadily improves as $q$ increases.

\begin{figure}[h]
	\centering
	\includegraphics[width=2.4in,height=2in]{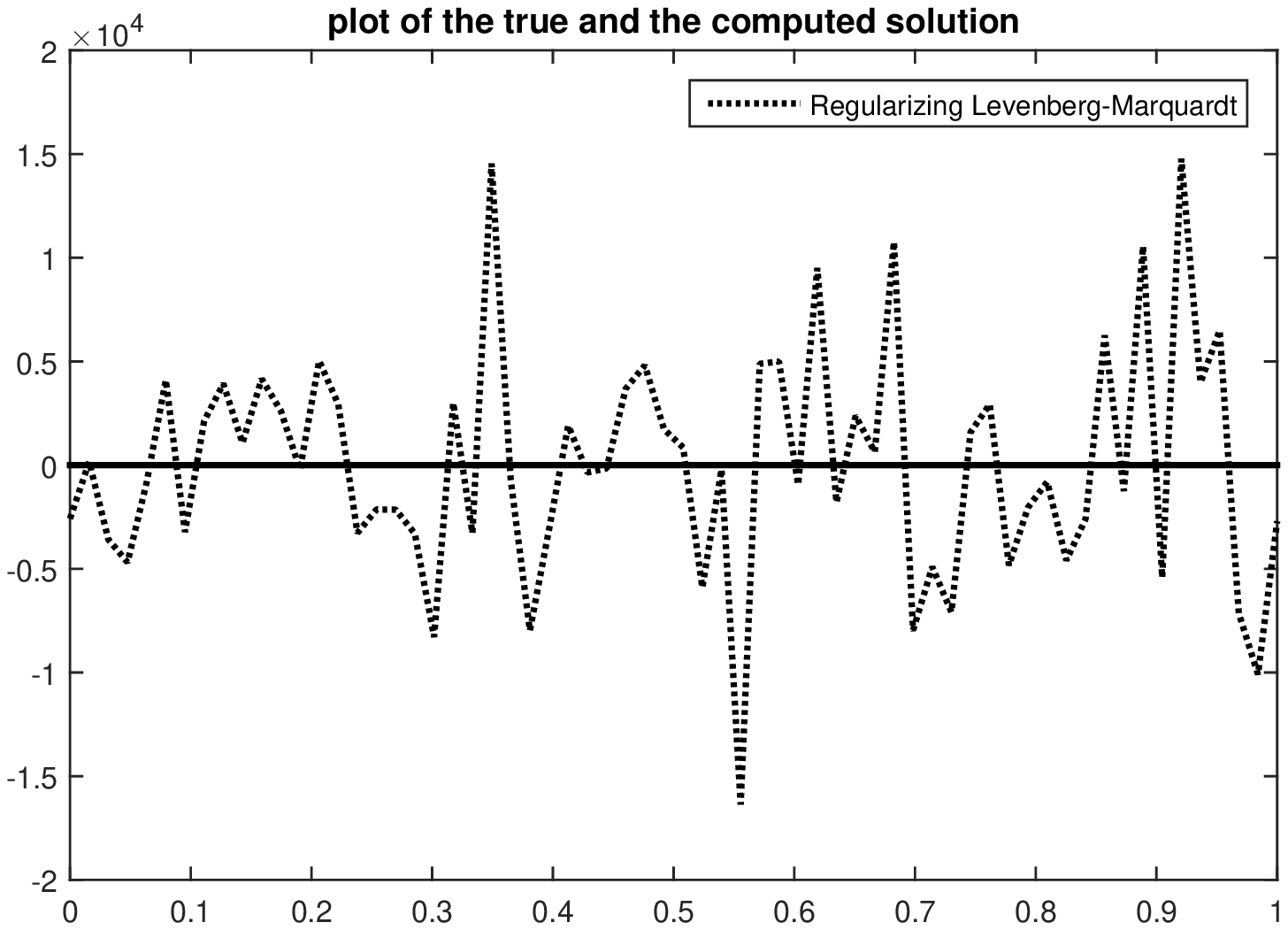}
	\includegraphics[width=2.4in,height=2in]{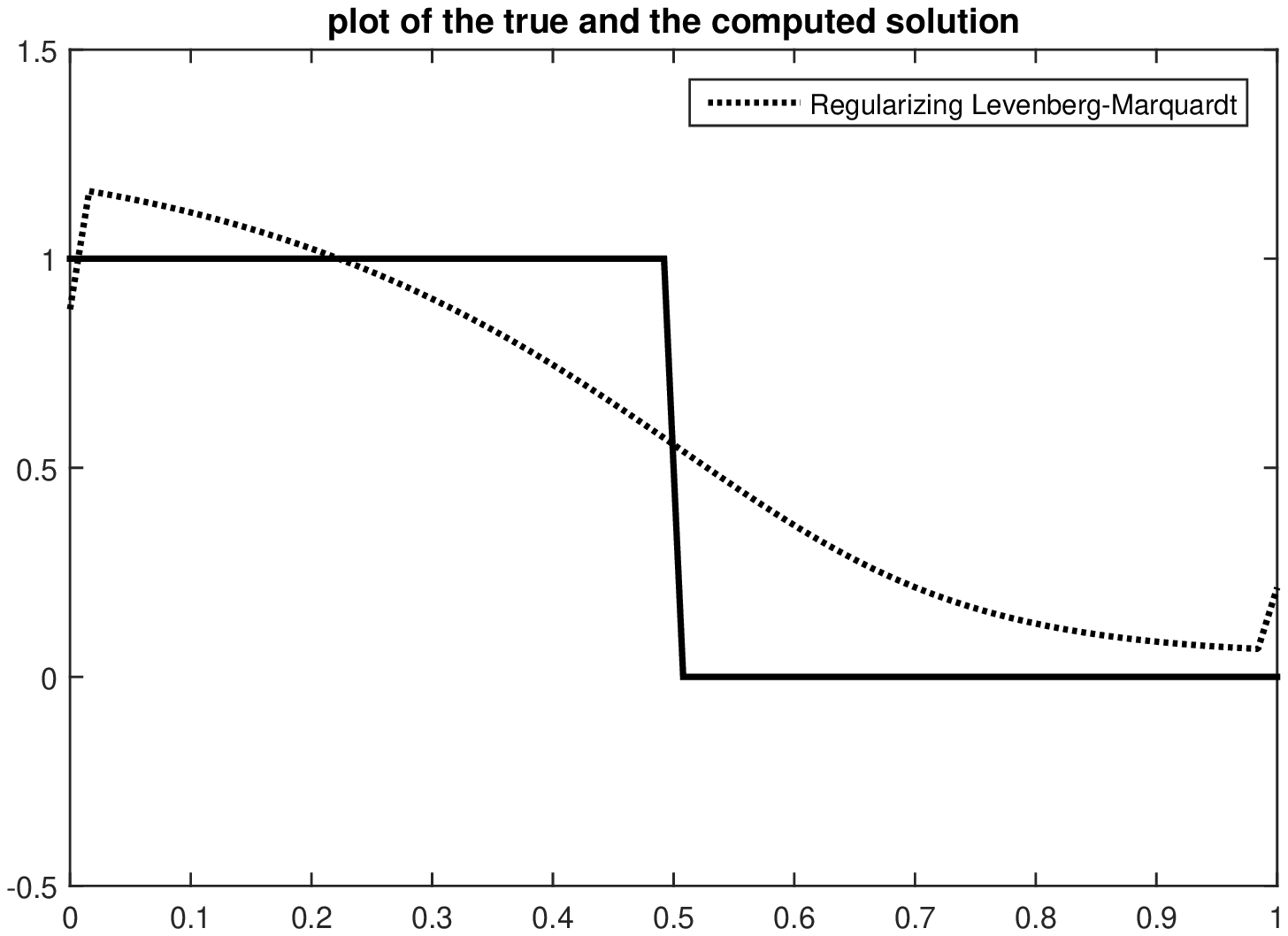}
	\includegraphics[width=2.4in,height=2in]{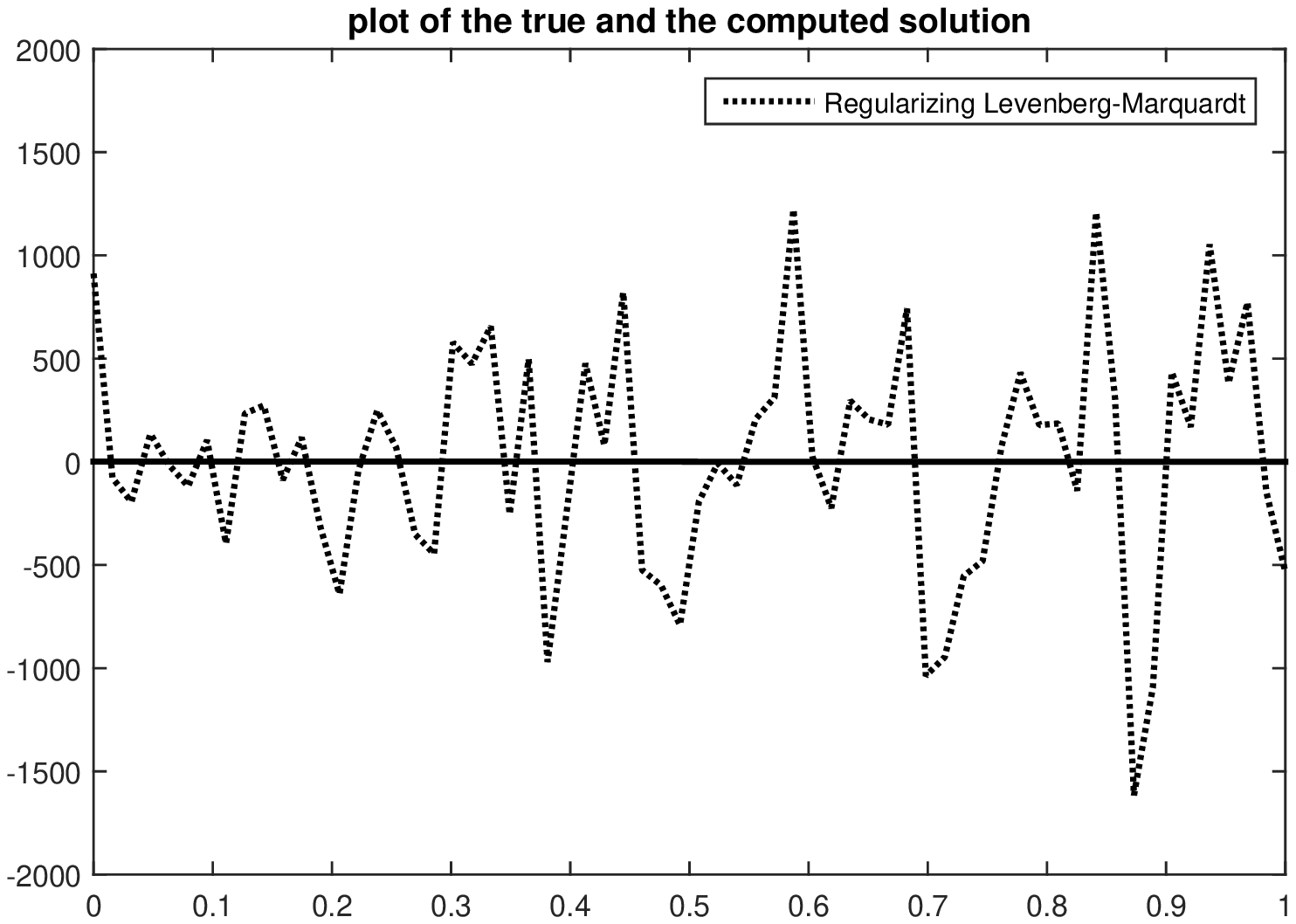}
	\includegraphics[width=2.4in,height=2in]{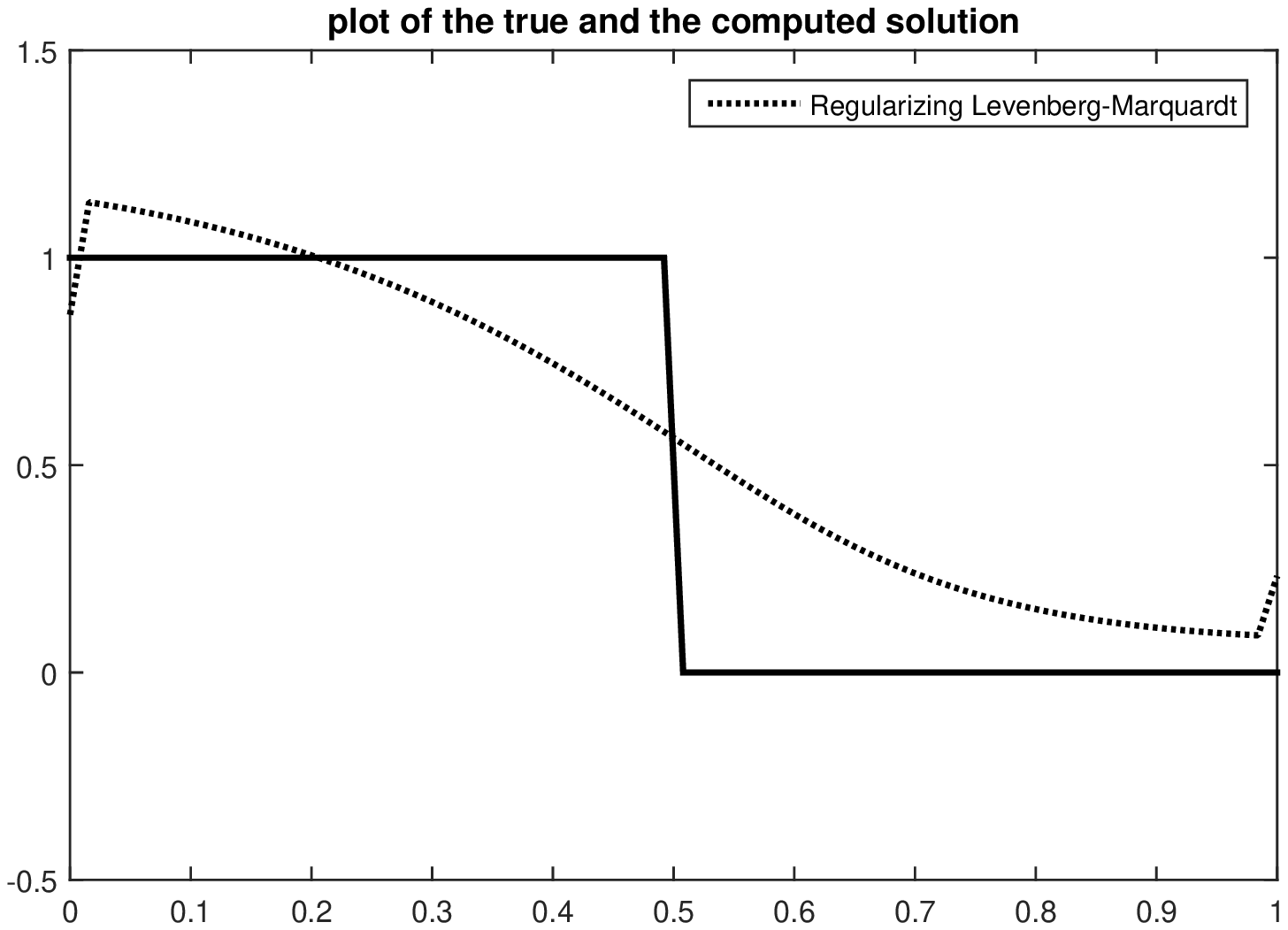}
	\caption{Problem P4, $\delta=10^{-2}$, $x_0=x_0(\beta,\chi)=x_0(1.5,0)$:
	approximate solution  computed by the regularizing Levenberg-Marquardt  
	method for values of $q=0.67,\, 0.70,\, 0.73,\, 0.87$. }
	\label{fig:LM_q}
\end{figure}

%
%

  \begin{figure}
  	\subfigure[]{\includegraphics[width=2.4in,height=2in]{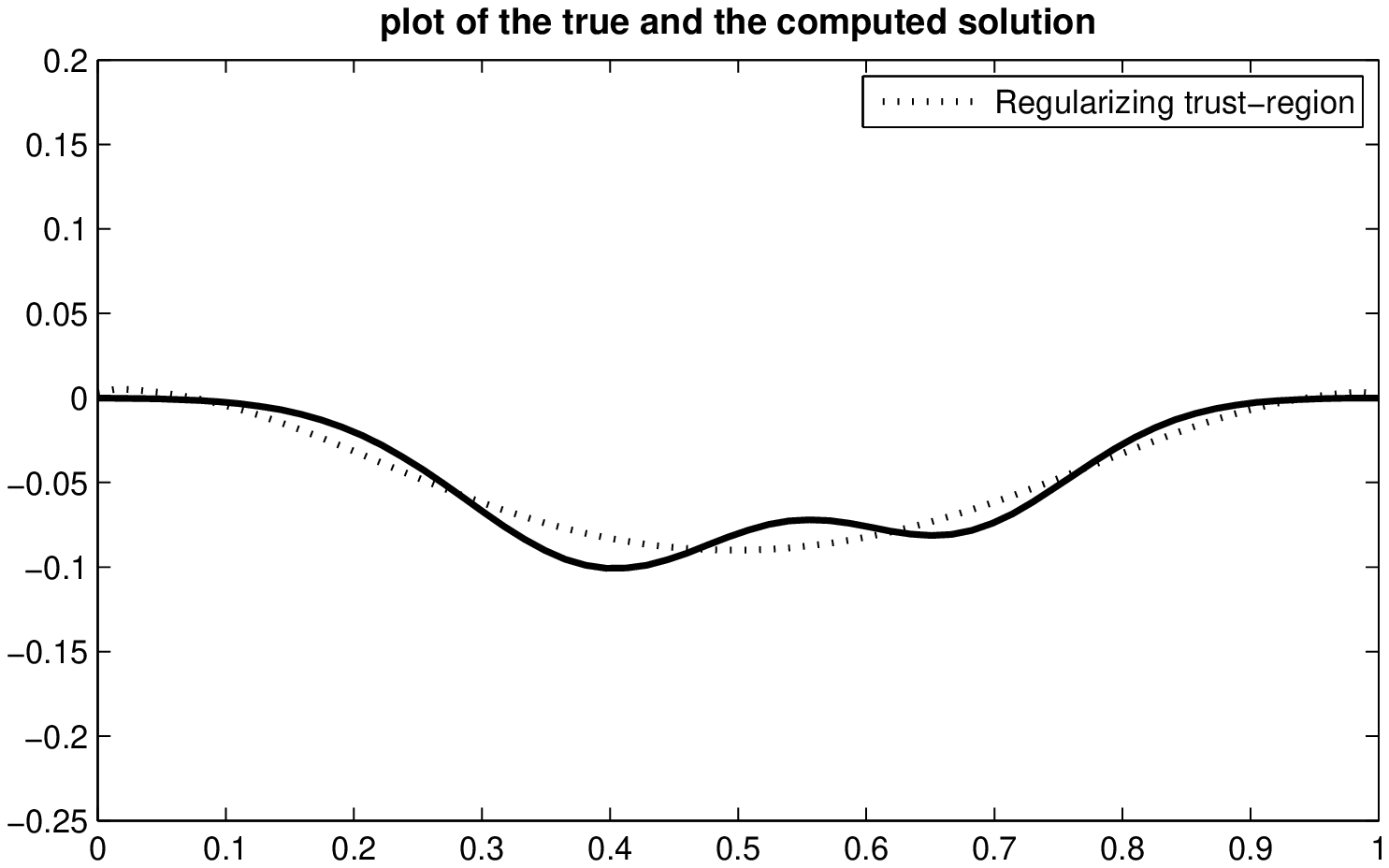}}
  	\subfigure[]{\includegraphics[width=2.4in,height=2in]{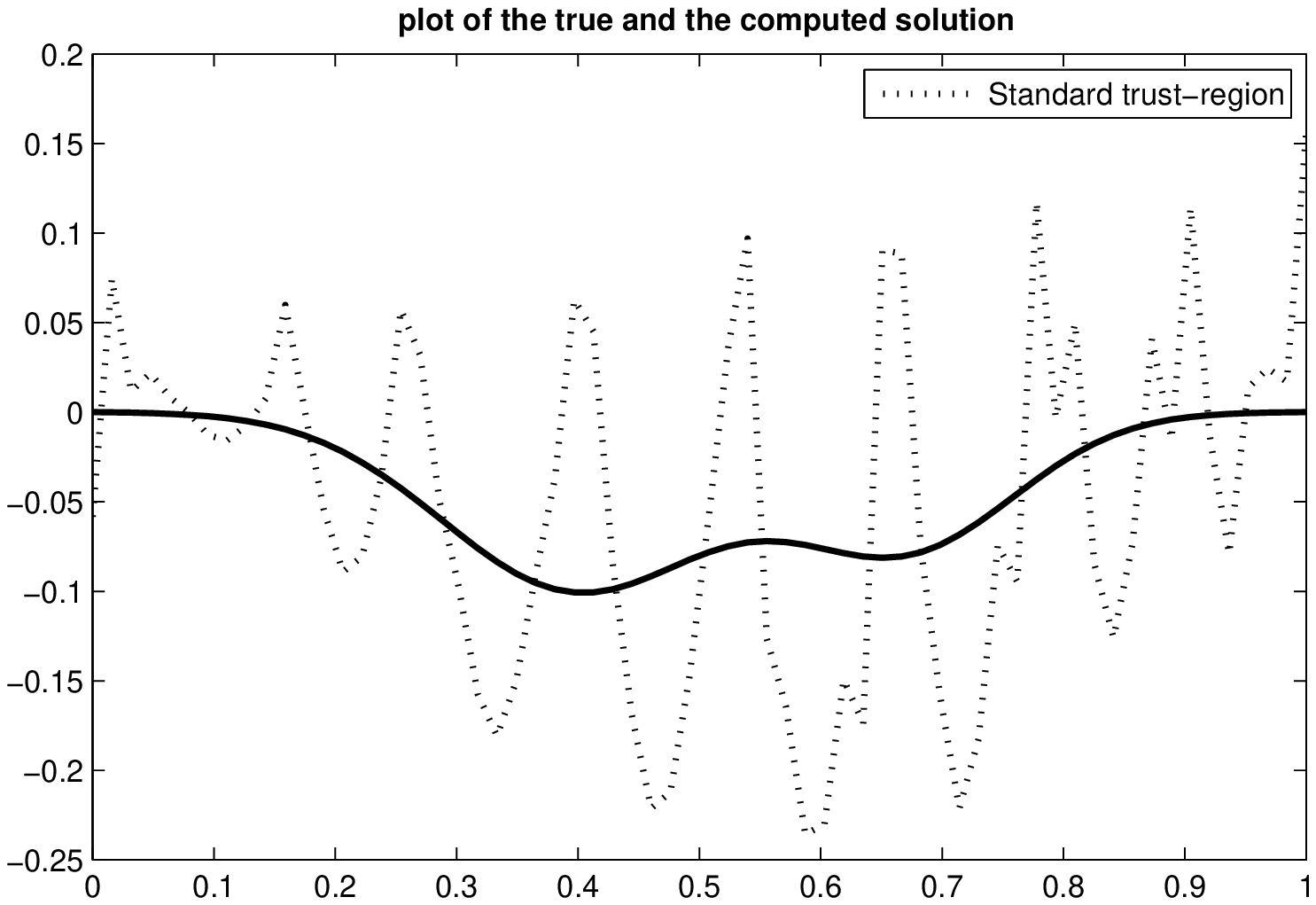}}
  	\subfigure[]{\includegraphics[width=2.4in,height=2in]{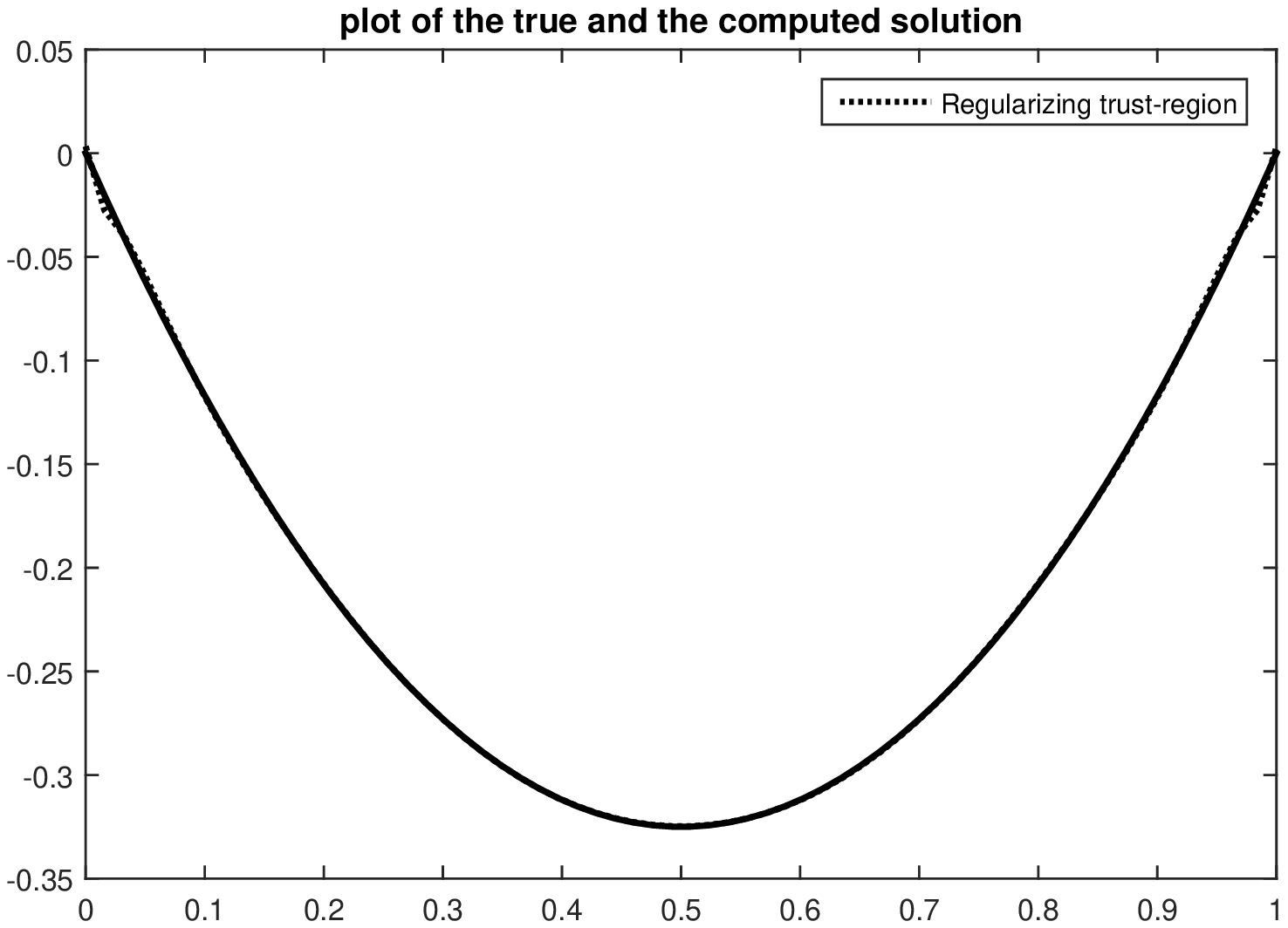}}
  	\subfigure[]{\includegraphics[width=2.4in,height=2in]{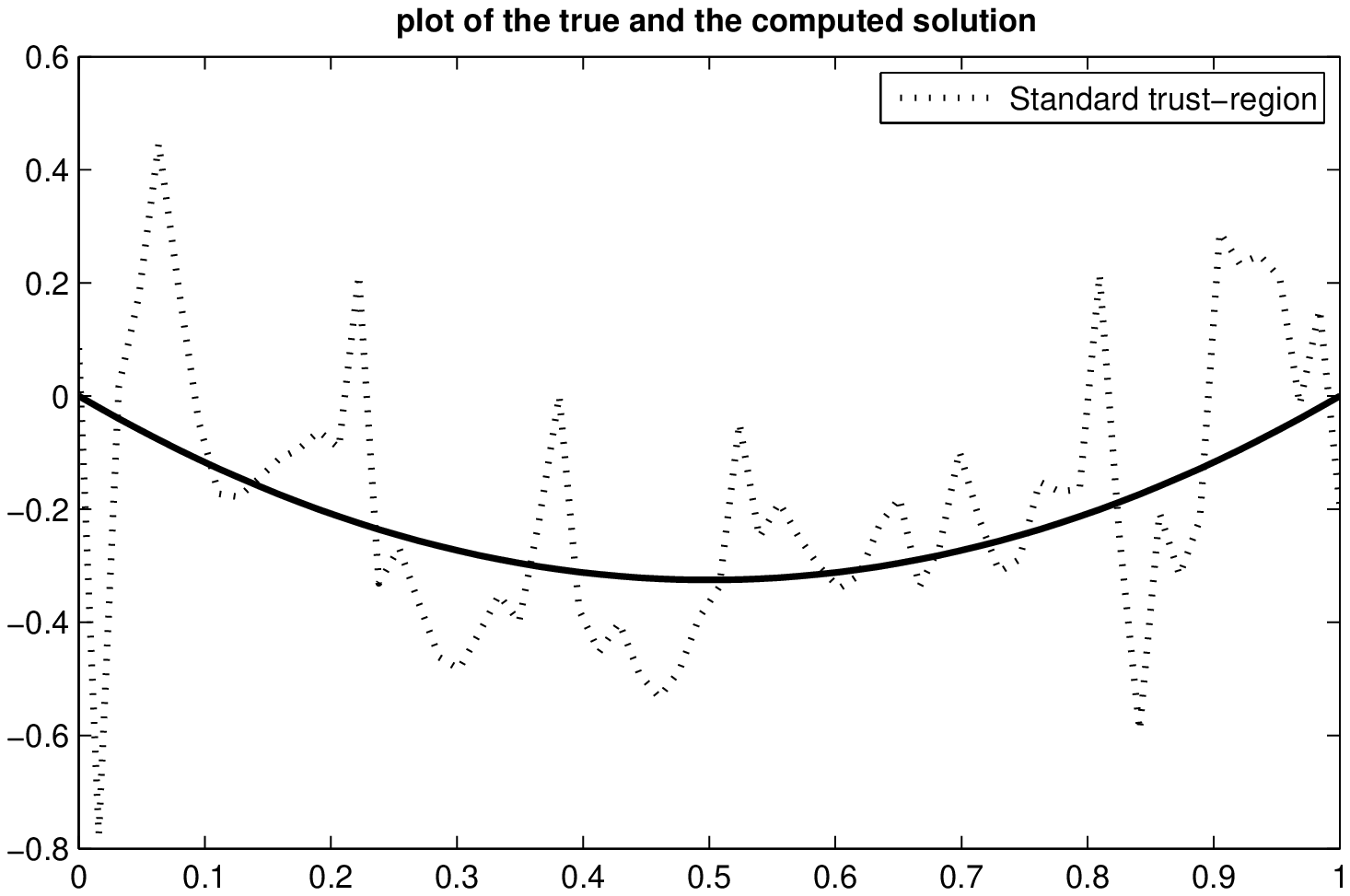}}
  	\subfigure[]{\includegraphics[width=2.4in,height=2in]{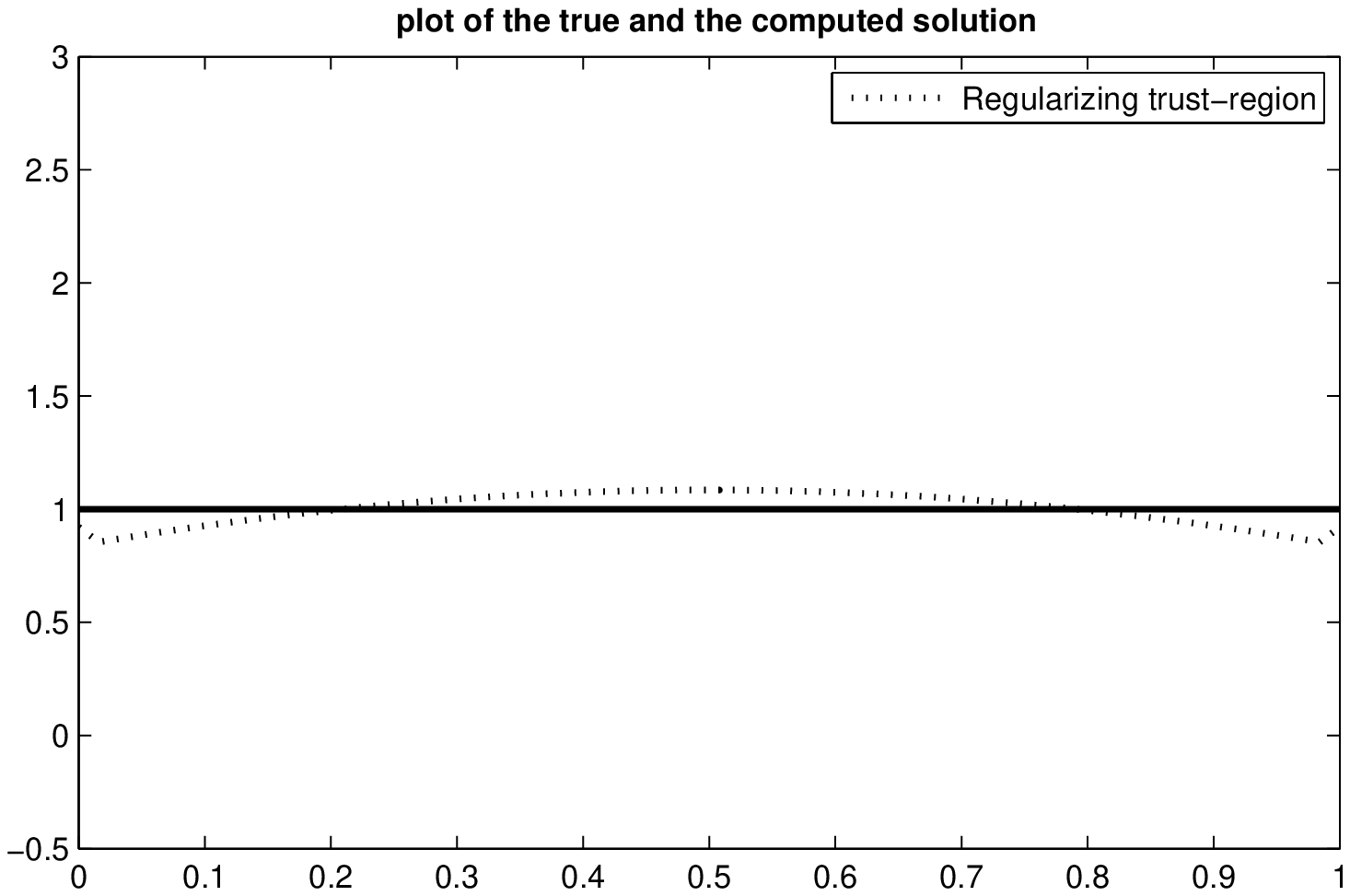}}
  	\subfigure[]{\includegraphics[width=2.4in,height=2in]{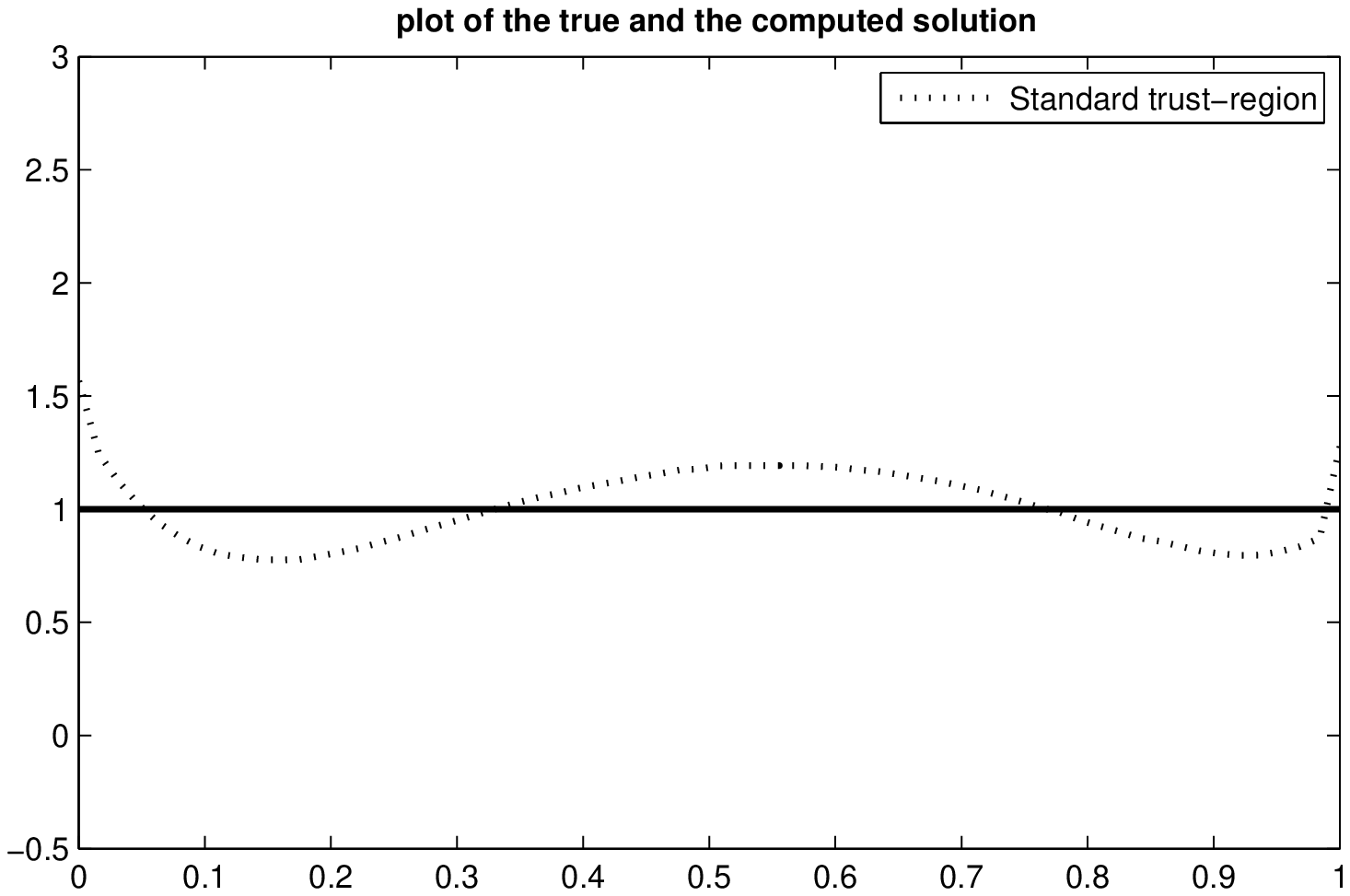}}
  	\subfigure[]{\includegraphics[width=2.4in,height=2in]{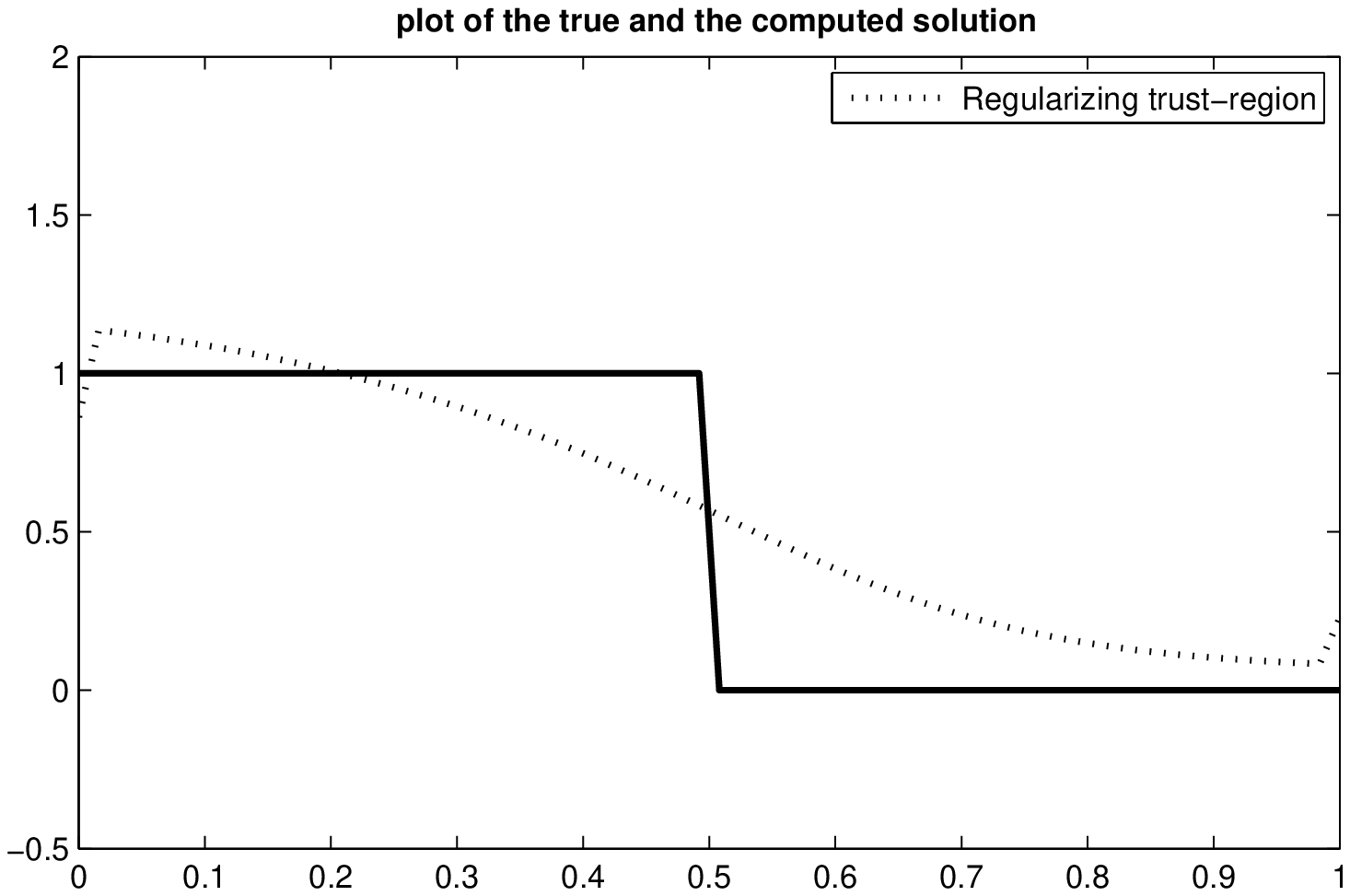}}
  	\subfigure[]{\includegraphics[width=2.4in,height=2in]{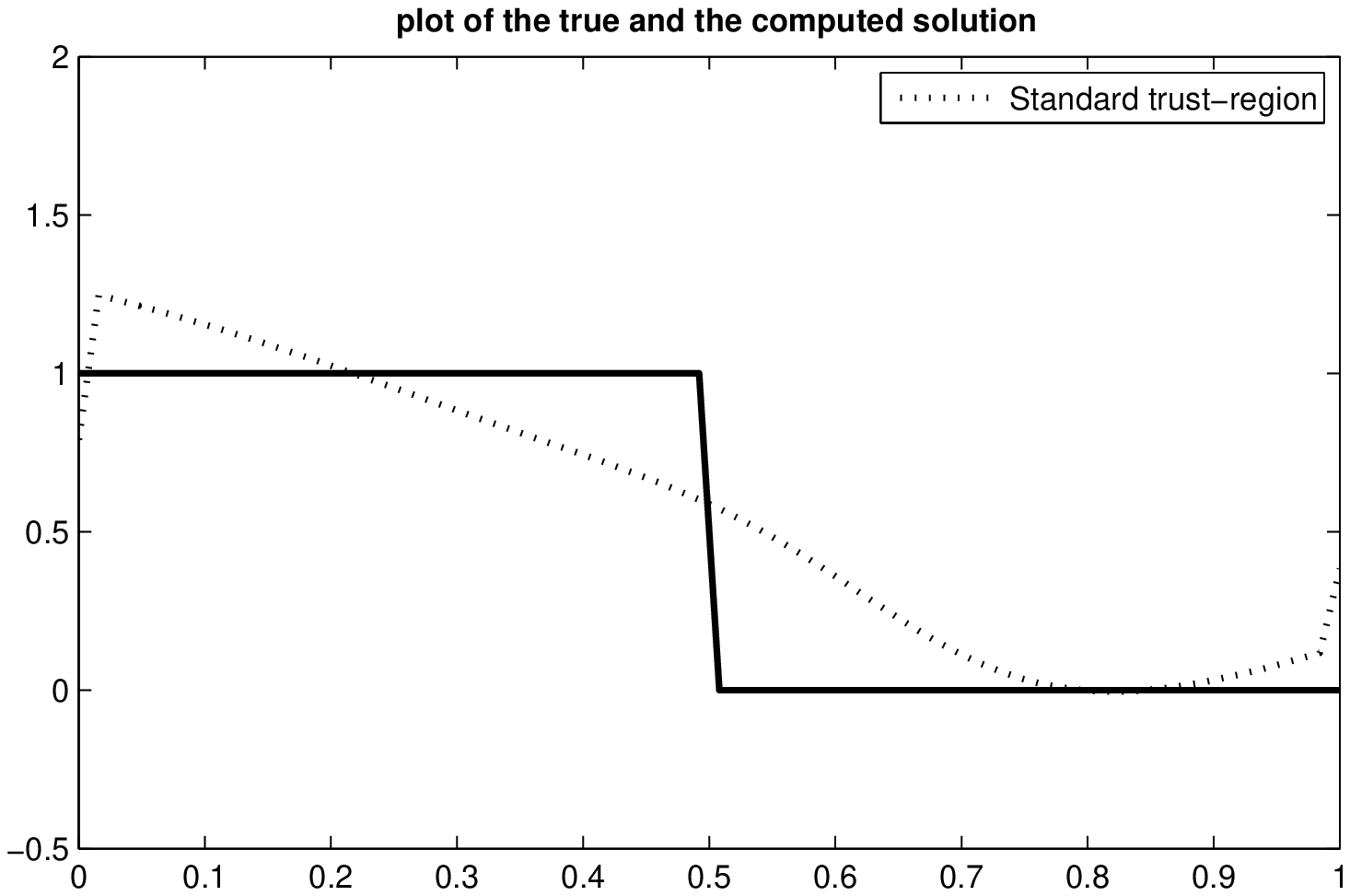}}
  	\caption{True solution (solid line) and approximate solutions (dotted line) computed by the regularizing  trust-region method (on the left) and the standard trust-region method (on the right). (a)-(b) problem P1,  $\delta=10^{-2}, x_0=0e$; (c)-(d) problem P2, $\delta=10^{-2}, x_0=0e$; (e)-(f) problem P3, $\delta=10^{-2}$, $x_0=x_0(1.25)$; (g)-(h) problem P4, $\delta=10^{-2}$, $x_0=x_0(0.5,0)$.}
  	
\label{fig:figure8}
  \end{figure}

We conclude this section considering the standard trust-region strategy.
It is well known that the standard  updating rule promotes
the use of inactive trust-regions, at least in the late stage of the procedure.
Clearly, this can adversely affect the solution of our test problems
and our experiments confirmed this fact. In particular, for $\delta=10^{-2}$ and 
 problems P1 and P2, the sequences computed by the standard trust-region method approach  solutions of the noisy problem. The same behaviour occurs in  most of the runs with P1 and P2 and noise level  $\delta=10^{-4}$. Conversely, the approximations provided by the regularizing trust-region procedure are  accurate approximations of true solutions in all the tests. 
Moreover, the approximations computed by the standard trust-region applied to problems P3 and P4   are less accurate than those computed by the regularizing trust-region although they do not show the  strong oscillatory behavior arising in problems P1 and P2. In problem P4, this behavior is evident when the second, third and fourth starting guesses are used, while the approximation computed starting from the first initial guess is as accurate as the one computed by the regularizing trust-region. This good result of the standard trust-region on problem P4 with $x_0=x_0(1,1)$ is  due to the fact that the trust-region is active in all iterations and therefore a regularizing behaviour is implicitely provided. As an example in Figure \ref{fig:figure8} we compare some solution approximations computed by the regularizing trust-region (left) and by the standard trust-region (right)  with  $\delta=10^{-2}$  applied to  problem P1  (figures (a)-(b)), P2 (figures (c)-(d)), P3 (figures (e)-(f)) and P4 (figures  (g)-(h)). 

\section{Conclusions}
 We have presented a trust-region  method for nonlinear ill-posed systems, possibly  with noisy data, where the regularizing behaviour is guaranteed by a suitable choice of the trust-region radius.
The proposed approach  shares the same local convergence properties as the regularizing Levenberg-Marquardt method
proposed by Hanke  in \cite{hanke} but it is more likely to satisfy th discrepancy principle irrespective of  the closeness of the initial guess  to 
a solution of (\ref{prob}).
The numerical experience presented confirms the effectiveness of the
trust-region radius adopted and the regularizing properties of  the resulting trust-region method. It also enlights that the new approach  is less sensitive than the  regularizing Levenberg-Marquardt method to the choice of the parameter $q$ in (\ref{seculare_q}) and (\ref{RQ}). Finally, as expected the numerical results show that the solution of the noisy problems may be misinterpreted by the  standard trust-region method.

\end{document}